\pdfoutput=1
\RequirePackage{ifpdf}
\ifpdf 
\documentclass[pdftex]{sigma}
\else
\documentclass{sigma}
\fi

\newcommand{\cC}{\mathcal{C}}
\newcommand{\cD}{\mathcal{D}}
\newcommand{\cG}{\mathcal{G}}
\newcommand{\cL}{\mathcal{L}}
\newcommand{\cF}{\mathcal{F}}
\newcommand{\cN}{\mathcal{N}}
\newcommand{\cW}{\mathcal{W}}
\newcommand{\R}{\texttt{R}}
\newcommand{\Q}{\texttt{Q}}

\def\dres{\partial{\rm Res}}
\def\ord{\operatorname{ord}}
\def\mod{{\rm mod}}
\def\rk{\operatorname{rk}}
\def\Ini{\operatorname{Ini}}

\numberwithin{equation}{section}

\newtheorem{Theorem}{Theorem}[section]
\newtheorem{Corollary}[Theorem]{Corollary}
\newtheorem{Lemma}[Theorem]{Lemma}
\newtheorem{Proposition}[Theorem]{Proposition}
 { \theoremstyle{definition}
\newtheorem{Definition}[Theorem]{Definition}
\newtheorem{Example}[Theorem]{Example}
\newtheorem{Remark}[Theorem]{Remark}
\newtheorem*{Algorithm}{Algorithm}
}

\begin{document}
\allowdisplaybreaks

\newcommand{\arXivNumber}{1902.01361}

\renewcommand{\thefootnote}{}

\renewcommand{\PaperNumber}{101}

\FirstPageHeading

\ShortArticleName{Commuting Ordinary Differential Operators and the Dixmier Test}

\ArticleName{Commuting Ordinary Differential Operators\\ and the Dixmier Test\footnote{This paper is a~contribution to the Special Issue on Algebraic Methods in Dynamical Systems. The full collection is available at \href{https://www.emis.de/journals/SIGMA/AMDS2018.html}{https://www.emis.de/journals/SIGMA/AMDS2018.html}}}

\Author{Emma PREVIATO~$^\dag$, Sonia L.~RUEDA~$^\ddag$ and Maria-Angeles ZURRO~$^\S$}

\AuthorNameForHeading{E.~Previato, S.L.~Rueda and M.A.~Zurro}

\Address{$^\dag$~Boston University, USA}
\EmailD{\href{mailto:ep@math.bu.edu}{ep@math.bu.edu}}

\Address{$^\ddag$~Universidad Polit\'ecnica de Madrid, Spain}
\EmailD{\href{mailto:sonialuisa.rueda@upm.es}{sonialuisa.rueda@upm.es}}

\Address{$^\S$~Universidad Aut\'onoma de Madrid, Spain}
\EmailD{\href{mailto:mangeles.zurro@uam.es}{mangeles.zurro@uam.es}}

\ArticleDates{Received February 04, 2019, in final form December 23, 2019; Published online December 30, 2019}

\Abstract{The Burchnall--Chaundy problem is classical in differential algebra, seeking to describe all commutative subalgebras of a ring of ordinary differential operators whose coefficients are functions in a given class. It received less attention when posed in the (first) Weyl algebra, namely for polynomial coefficients, while the classification of commutative subalgebras of the Weyl algebra is in itself an important open problem. Centralizers are maximal-commutative subalgebras, and we review the properties of a basis of the centralizer of an operator $L$ in normal form, following the approach of K.R.~Goodearl, with the ultimate goal of obtaining such bases by computational routines. Our first step is to establish the \texttt{Dixmier test}, based on a lemma by J.~Dixmier and the choice of a suitable filtration, to give necessary conditions for an operator $M$ to be in the centralizer of $L$. Whenever the centralizer equals the algebra generated by $L$ and~$M$, we call $L$, $M$ a Burchnall--Chaundy~(BC) pair. A construction of BC pairs is presented for operators of order $4$ in the first Weyl algebra. Moreover, for true rank~$r$ pairs, by means of differential subresultants, we {\it effectively compute} the fiber of the rank $r$ spectral sheaf over their spectral curve.}


\Classification{13P15; 14H70}

\renewcommand{\thefootnote}{\arabic{footnote}}
\setcounter{footnote}{0}

\section{Introduction}\label{sec-Introduction}

In the 1923 seminal paper by Burchnall and Chaundy \cite{BC}, the authors proposed to describe all pairs of commuting differential operators that are not simply contained in a polynomial ring~${\bf C}[M]$\footnote{Although the field of coefficients is not mentioned in~\cite{BC}, we work over the complex numbers~${\bf C}$ in this paper unless
otherwise specified.}, $M\in\cD$ (cf.\ Section~\ref{sec-Preliminares} below for
notation). We note that, whenever two differential operators $A$, $B$, commute
with an operator $L$ of order greater than zero, then they commute with each
other (cf.\ Corollary~\ref{cor-maxcom}), and therefore
maximal-commutative subalgebras of $\cD$ are centralizers; these are the main
objects we seek to classify. In addition,
we will always assume that a commutative subalgebra
 contains a normalized element
$L=\partial^n+u_{n-2}\partial^{n-2}+\cdots +u_0$, although
some proviso is needed (cf., e.g., \cite{BZ}), except in the `formal' case
when the coefficients are
just taken to be formal power series.
 We will say that the {\it Burchnall--Chaundy $($BC$)$ problem} asks when the
centralizer ${\mathcal C}_{\mathcal D}(L)$ of an operator $L$ is not a polynomial ring
(which we regard as a `trivial' case, for example ${\bf C}[G]$,
{with $L$ a~power of some $G\in\mathcal{D}$}) and we call such an $L$
a~``BC solution''. Burchnall and Chaundy immediately make the observation that if the orders of
two commuting $L$ and $B$ are coprime, then either one is a BC solution.
Eventually~\cite{BC1}, they were able to classify the commutative subalgebras
${\bf C}[L,B]$ of rank one~-- the rank, defined in
Section \ref{sec-TrueRk} for any subset of $\cD$, is the greatest common
divisor of the orders of all elements of ${\bf C}[L,B]$.
The classification problem is wide open in higher (than one) rank, although a~theoretical geometric description was given~\cite{K78, Mu2}.

In 1968 Dixmier gave an example \cite{Dix} of BC solution: he showed that for any complex number $\alpha$ in ${\bf C}$ the differential operators
\begin{gather}\label{eq-Dixintro}
 L=H^2+2x\qquad \text{and} \qquad B=H^3+\frac{3}{2}(xH+Hx),\qquad \mbox{with} \quad H={\partial^2+x^3+\alpha}
\end{gather}
identically satisfy the algebraic equation $B^2=L^3-\alpha$, and moreover, that the algebra ${\bf C}[L,B]$ is a~maximal-commutative subalgebra of the first Weyl algebra $A_{1}({\bf C})$, since it is the centrali\-zer~$\cC(L)$ of the operator $L$ in $A_{1}({\bf C})$, thus providing the first example of BC solution with ${\bf C}[L,B]$ of higher rank\footnote{In fact, the rank is two, but under the antiautomorphism of $A_1({\bf C})$ that interchanged differentiation and independent variable, ${\bf C}[L,B]$ also provides an example of rank three.} provided $\alpha\ne 0$.

To give a rough idea of the difference between rank one and higher, we recall that centra\-li\-zers~${\mathcal C}_{\mathcal D}(L)$ have quotient fields that are function fields of one variable, therefore can be seen as affine rings of curves, and in a formal sense these are {\it spectral curves}. Burchnall and Chaundy's theory for rank one shows that the algebras that correspond to a fixed curve make up the (generalized) Jacobian of that curve, and the $x$ flow is a holomorphic vector field on it. We may (formally) view this
as a ``direct'' spectral problem; the ``inverse'' spectral problem allows us to reconstruct the coefficients of the operators (in terms of theta functions) from the data of a~point on the Jacobian (roughly speaking, a~rank-one sheaf on the curve). The case of rank $r>1$ corresponds to a vector bundle of rank $r$ over the spectral curve: there is no
explicit solution to the ``inverse'' spectral problem (despite considerable progress achieved in \cite{KN1, KN2,KN2+}), except for the case of elliptic spectral curves; we will refer to some of the relevant literature below, but we will not attempt at completeness because our goal here is narrower, and the higher-rank literature is quite hefty.

We now describe the goals and results of this paper. There are several properties, relevant to the classification and explicit description of commutative subalgebras, both in the case of ${\mathcal D}$ and of $A_{1}(\mathbf{C})$, that are difficult to discern: our plan is to address them with the aid of computation.

First, a centralizer $\cC_{\cD}(L)$ is known to be a finitely generated free ${\bf C}[L]$-module and we use a~result by Goodearl in \cite{Good} to the effect that the cardinality of any basis is a~divisor of $n=\ord(L)$. By restricting attention to polynomial coefficients, in Section~\ref{subsect-grading} we determine the initial form of the elements in the centralizer of $L$, by automating the ``\texttt{Dixmier test}'' by means of a~suitable filtration. As a~consequence, we can guarantee in Section~\ref{sec-cen4} that the centralizer of an operator of order~$4$ in the first Weyl algebra $A_{1} ({\bf C})$ is the ring of a~{\it plane} algebraic curve in ${\bf C}^2$ (this, given that all centralizers are affine rings of irreducible, though not necessarily reduced, curves, amounts to saying that there is a~plane model of the curve which only misses one
smooth point at infinity, cf.~\cite{Mu1}, or equivalently, that the centralizer can be generated by two elements).

\looseness=-1 Additionally, given a differential operator $M$ that commutes with $L$, we have the sequence of inclusions ${\bf C}[L]\subseteq {\bf C}[L,M] \subseteq \cC_{\cD}(L)$ and all of them could be strict. In this paper we are interested in testing, again for polynomial coefficients, whether a differential operator $B$ exists such that $\cC_{\cD}(L)$ equals ${\bf C}[L,B]$. In such case we call $L$, $B$ a~``Burchnall--Chaundy (BC) pair'' and~{$\cC_{\cD}(L)$ will be the free ${\bf C}[L]$-module with basis $\{1,B\}$}, as a~consequence of Goodearl's theory~\cite{Good}, cf.\ Section~\ref{sec-Dixmier}. Given an operator $M$ in the centralizer of $L$, we give a~procedure to decide if $M$ belongs to ${\bf C}[L]$, that is ${\bf C}[L]={\bf C}[L,M]$; this \textit{``triviality test''} can be performed by means of the differential resultant, see Section~\ref{sec-contruction}. Next, to the question whether $L$, $M$ is a~BC pair, we give an answer for operators $L=L_4$ of order $4$ in $A_1({\bf C})$. Moreover we design an algorithm, ``\texttt{BC pair}'' in Section~\ref{sec-contruction}, that given a commuting pair~$L_4$,~$M$ returns a~BC pair~$L_4$,~$B$. Our algorithm relies on a construction given in Section~\ref{sec-contruction} and its accuracy is guaranteed by Theorem~\ref{thm-al-2}. By means of iterated Euclidean divisions it produces a~system of equations whose solution allows reconstruction of a good partner $B$ such that $L$, $B$ is the desired BC pair. Explicit examples of the performance of this construction are given in Section~\ref{sec-contruction}.

Another issue is that of ``true'' vs.\ ``fake'' rank; this will be defined in more detail, with examples, in Section~\ref{sec-TrueRk}. Here we briefly say that a pair $L$, $M$ of commuting operators whose orders are both divisible by $r$, is called a ``true rank~$r$ pair'' if $r$ is the rank of the algebra~$ {\bf C}[L,M]$. We prove in Theorem~\ref{prop-BCtrue} that BC pairs are true-rank pairs. Of course, not every true-rank pair is a BC pair and, in the process of searching for new true-rank pairs, by means of Gr\" unbaum's approach~\cite{Grun}, one obtains families of examples, see Example~\ref{ex4}. One of our goals is to give true rank~$r$ pairs and important contributions were made by Grinevich~\cite{Gri}, Mokhov~\cite{Mo82,Mo90,Mo1,Mo2}, Mironov~\cite{Mi1}, Davletshina and Shamaev~\cite{DS}, Davletshina and Mironov~\cite{DM}, Mironov and Zheglov~\cite{Mi2,MZ2014}, Oganesyan~\cite{Og2016, Og2017, Og2018}, Pogorelov and Zheglov~\cite{PZ}. To check our results we constructed new true rank~$2$ pairs, by means of non self-adjoint operators of order~$4$ with genus~$2$ spectral curves, see Examples~\ref{ex1} and~\ref{ex3}.

Lastly, for commuting pairs $L$, $M$, it is easy to observe the existence of a~polynomial $h(\lambda,\mu)$ with constant coefficients such that,
identically in the independent variable, $h(L,M)=0$: Burchnall and Chaundy showed that the opposite is also true~\cite{BC, BC2}. This is the defining polynomial of a plane curve, commonly known as spectral curve $\Gamma$, and it can be computed by means of the differential resultant of $L-\lambda$ and $M-\mu$. Furthermore for a true rank~$r$ pair we have
\begin{gather*}
\dres (L-\lambda,M-\mu)=h(\lambda,\mu)^r,
\end{gather*}
see for instance \cite{Prev, W}. By means of the subresultant theorem~\cite{Ch}, we prove in Section~\ref{sec-gcd}, Theorem~\ref{thm-gcd}:
Given a true rank~$r$ pair~$L$, $M$, the greatest common (right) divisor for $L-\lambda_0$ and $M-\mu_0$ at any point $P_0=(\lambda_0,\mu_0)$ of $\Gamma$
is equal to the $r$th differential subresultant $\cL_r(L-\lambda_0,\allowbreak M-\mu_0)$,
and is a differential operator of order~$r$. In this manner we obtain an
{\it explicit} presentation of the right factor of order~$r$ of $L-\lambda_0$
and $M-\mu_0$ that can be {\it effectively computed}. Hence an {\it explicit}
description of the fiber $\cF_{P_0}$ of the rank $r$ spectral sheaf $\cF$ in
the terminology of~\cite{BZ,PW}, where the operators are given in
the ring of differential operators with coefficients in the formal power
series ring ${\bf C}[[x]]$. The factorization of ordinary differential operators using differential subresultants, for non self-adjoint operators, is an important contribution of this work.

Explicit computations for true rank~$2$ self-adjoint and non self-adjoint operators in the
first Weyl algebra $A_1({\bf C})$ are shown in Sections~\ref{sec-gcd} and~\ref{sec-contruction}. We
use these examples to show the performance of our effective
results. Although at this stage we have only implemented our project for rank
two, this is the first step in which complete explicit results were available
(cf.~\cite{Grun}), but we believe that our computational approach to the set
of issues we described has the potential to streamline the theory and
be extended to any rank. We note, without attempting at complete references,
that in rank three Gr\"unbaum's work was extended by Latham
(cf., e.g.,~\cite{LP}) and Mokhov~\cite{Mo1, Mo2} (independently); as for the Weyl algebra, cf.\ the references we gave above.
Computations were carried with Maple~18, in particular using the package OreTools.

\section{Preliminaries}\label{sec-Preliminares}

We are primarily interested in the ring of differential operators ${\mathcal D}$, but it is useful to view it as a subring of the ring of formal pseudodifferential operators $\Psi$, namely the set
\begin{gather*}
\Psi =\left\{
\sum_{j=-\infty}^N u_j(x)\partial^j,\ u_j\ \text{\rm analytic in
some connected neighborhood of}\ x=0\right\}.
\end{gather*}
If we think of these symbols as acting on functions of $x$ by multiplication and differentiation: $(u (x)\partial )f(x)=u\frac{{\rm d}}{{\rm d}x}f$, and formally integrate by parts: $\int (uf^\prime)=u f-\int (u^\prime f)$, we can motivate the composition rules
\begin{gather*}
\partial u =u\partial +u^\prime,\\
\partial^{-1}u =u\partial^{-1}-u^\prime \partial^{-2}+u^{''}\partial^{-3}-\cdots
\end{gather*}
and easily check an extended Leibnitz rule for $A,B\in\Psi$:
\begin{gather*}A\circ B=\sum_{i=0}^\infty \frac{1}{i!}\tilde\partial^i A\ast\partial^i B,\end{gather*}
where $\tilde\partial$ is a partial differentiation w.r.t.\ the symbol
$\partial$ and $\ast$ has the effect of bringing all functions to the
left and powers of $\partial$ to the right. Observe that the first Weyl algebra $A_1 ({\bf C})$ is a subring of the ring of differential operators ${\bf C}(x)[\partial]$ with $\partial=\partial=\partial/\partial x$ and $[\partial,x]=1$. Hence a subring of~$\Psi$.

The differential ring $\Psi$ contains the differential subring ${\mathcal D}$ of differential operators $A=\sum\limits_0^N u_j\partial^j$ and we denote by $(\ )_{+}$ the projection $B_{+}=\sum\limits_0^N u_j\partial^j$ where $B=\sum\limits_{-\infty}^N u_j\partial^j$.

We also see that if $L$ has order $n>0$ and its leading coefficient is regular, i.e., $u_n(0)\not= 0$, then $L$ can be brought to standard form
\begin{gather*}
 L=\partial^n +u_{n-2}(x)\partial^{n-2}+u_{n-3}(x)\partial^{n-3}+\cdots +u_0(x)
\end{gather*}
by using change of variable and conjugation by a function, which are the only two automorphisms of ${\mathcal D}$; we shall always assume $L$ to be in standard form, i.e., $u_1 (x)=0$. We note that in~\cite{BZ}, for completeness, the authors recall a(n essentially formal) proof of the facts we mentioned, to bring $L$ into standard form.

\begin{Remark} The coefficients $u_j(x)$ in the definition of $\Psi$ are often required to be analytic functions near $x=0$, because the algebro-geometric constructions preserve this restriction; typically, statements of differential algebra hold formally, and in particular, our results are mostly concerned with polynomial coefficients, therefore we do not aim at complete generality. Analytic/formal cases of the ring $\Psi$ are treated in~\cite{SS}, with emphasis on certain types of modules over $\Psi$.
 \end{Remark}

\subsection{Centralizers for ODOs}

Unless otherwise specified, we will work with a differential field $(K,\partial)$, with field of constants the field of complex numbers ${\bf C}$, and the ring of differential operators $\cD=K[\partial]$. Given a differential operator $L$ in $\cD$ in standard form, we denote its centralizer in $\cD$ as
\[\cC_{\mathcal D}(L)=\{M\in \cD\,|\, [L,M]=0\}.\]

We recall the reason why centralizers are maximal-commutative subalgebras of~$\mathcal{D}$. We cite two lemmas~\cite{Ve}, the first being straightforward to check; the second is proved in~\cite{Ve} by a~beautiful Lie-derivative argument.

\begin{Lemma}[\cite{Ve}] If $A=a_n\partial^n+a_{n-1}\partial^{n-1}+\cdots +a_0$, $B=b_m\partial^m+b_{m-1}\partial^{m-1}+\cdots +b_0\in {\mathcal D}$ are such that $n>0$ and $\ord [A,B]<n+m-1$, then $\exists\,\alpha\in {\bf C}$ s.t.~$b_m^n=\alpha a_n^m$. Moreover, if $a_n$ and $b_m$ are constant and $\ord[A,B]<n+m-2$, then $\exists\,\alpha,\beta\in {\bf C}$ s.t.~$b_{m-1}=\alpha a_{n-1}+\beta$.
\end{Lemma}

\begin{Lemma}[\cite{Ve}]If ${\mathcal A}$ is a commutative subalgebra of ${\mathcal D}$ and $M\in {\mathcal D}$, $\exists\, p\in {\bf Z}\cup\{ -\infty\}$ s.t.~$\forall\, L\in {\mathcal A}$, $\ord L>0$, $\ord [M,L]=p+ \ord L$.
\end{Lemma}

\begin{Corollary}\label{cor-maxcom} If $\ord L>0$ and $A,B\in {\mathcal D}$ both commute with $L$, then $[A,B]=0$; in particular, ${\mathcal C}_{\mathcal D}(L)$ is commutative, hence every maximal-commutative subalgebra of~${\mathcal D}$ is a centralizer.
\end{Corollary}

\begin{Remark} The analog of Corollary~\ref{cor-maxcom} is not true for operators on finite-dimensional spaces (it is easy to find two noncommuting matrices that commute with a third one).
\end{Remark}

In $\Psi$ any (normalized) $L$ has a unique $n$th root, $n=\ord L$, of the form
\begin{gather*}
 {\mathcal L}=\partial +u_{-1}(x)\partial^{-1}+u_{-2}(x)\partial^{-2}+\cdots .
\end{gather*}

The next result can be shown by using the fact that $\Psi$ is a graded ring.
\begin{Theorem}[I.~Schur, \cite{Sch}]
\[
{\mathcal C}_{\mathcal D}(L)=\left\{ \sum_{-\infty}^N c_j{\mathcal L}^j,\,c_j\in {\bf C}\right\}\cap {\mathcal D}.
\]
\end{Theorem}

\begin{Remark}\label{locpar} Schur's theorem shows that the quotient field of ${\mathcal C}_{\mathcal D}(L)$ is a function field of one variable; indeed, a $B$ which commutes with $L$ must satisfy an algebraic equation $f(L,B)=0$ (identically in $x$), by a~dimension count as sketched in \cite{Mum2}, moreover the degree of $f$ in $B$ is bounded; Burchnall and Chaundy show the existence of $f(L,B)$ by using the dimension of the vector space of common eigenfunctions of $L-\lambda_0$ and $B-\mu_0$ for a pair $(\lambda_0,\mu_0)$ such that $f(\lambda_0, \mu_0 )=0$. We will use this idea to give the equation of the curve algorithmically. Schur's point of view has the advantage that $\mathcal{L}$ can be viewed as the inverse of an (analytic) local parameter~$z$ at the point at infinity of the curve defined by $f(\lambda ,\mu )=0$ on the affine $(\lambda , \mu )$-plane. Think of an eigenfunction $\psi$ of $\mathcal{L}=\partial$ as ${\rm e}^{kx}$; the differential operators in $\mathcal{C}_\Psi (\mathcal{L})$ act on $\psi$ as polynomials in~$k$, and correspond to the affine ring of the spectral curve. The non-trivial case is achieved by conjugating with the ``Sato opertor'', $S^{-1}\partial S=\mathcal{L}$; this equation can be solved formally for any normalized $\mathcal{L}$.
\end{Remark}

\subsection{True rank}\label{sec-TrueRk}

The {\it rank} of a subset of ${\mathcal D}$ is the greatest common divisor of the orders of all the elements of that subset. However, we are mainly interested in the rank of the subalgebra generated by the subset. In particular, given commuting differential operators~$L$ and~$M$, let us denote by $\rk(L,M)$ the rank of the pair, which we will compare with the rank $\rk({\bf C}[L,M])$ of the algebra ${\bf C}[L,M]$ they generate.

A polynomial with constant coefficients satisfied by a commuting pair of differential operators is called a {\it Burchnall--Chaundy $($BC$)$ polynomial}, since the first result of this sort appeared is the $1923$ paper~\cite{BC} by Burchnall and Chaundy. In fact, they showed that the converse is also true, namely if two (non-constant) operators satisfy identically a~polynomial in two indeterminates~$\lambda$,~$\mu$ that belongs to $\mathbf{C}[\lambda ,\mu ]$, then they commute.

Let us assume that $n=\ord(L)$ and $m=\ord(M)$. The idea is that by commutativity $M$ acts on $V_\lambda$, the $n$-dimensional vector space of solutions $y(x)$ of $Ly=\lambda y$ ($L$ is regular); $f(\lambda,\mu )$ is the characteristic polynomial of this operator; to see that $f(L,M)\equiv 0$ it is enough to remark that $f(\lambda,\mu )=0$ iff~$L$,~$M$ have a~``common eigenfunction'':
\begin{gather*}
Ly=\lambda y,\\
By=\mu y,
\end{gather*}
hence $f(L,M)$ would have an infinite-dimensional
kernel (eigenfunctions belonging to distinct eigenvalues
$\lambda_1,\ldots,\lambda_k$ are independent by a Vandermonde
argument).

What brings out the algebraic structure of the problem,
and of the polynomial $f$, is the construction of the Sylvester matrix $S_0(L,M)$. This is the coefficient matrix of the extended system of differential operators
\begin{gather}\label{eq-extended system}
\Xi_0(L,M)=\big\{\partial^{m-1}L,\ldots, \partial L, L, \partial^{n-1}M,\ldots ,\partial M, M\big\}.
\end{gather}
Observe that $S_0(L,M)$ is a squared matrix of size $n+m$ and entries in~$K$. We define the {\it differential resultant} of $L$ and $M$ to be
$\dres(L,M):=\det (S_0(L,M))$. For a recent review on differential resultants see~\cite{McW}.
It is well known that
\begin{gather}\label{eq-res}
f(\lambda,\mu)=\dres(L-\lambda,M-\mu)
\end{gather}
is a polynomial with constant coefficients satisfied by the operators $L$ and $M$, see~\cite{Prev,W}. Moreover the plane algebraic curve $\Gamma$ in~${\bf C}^2$ defined by $f(\lambda,\mu)=0$ is known as the {\it spectral curve}~\cite{BC}.

\begin{Remark}Since the algebra ${\bf C}[L,B]$ has no zero-divisors, it can be viewed as the affine ring ${\bf C}[\lambda,\mu]/(h)$ of a plane curve, with $h(\lambda,\mu)$ an irreducible polynomial. The BC curve $=\{ (\lambda,\mu )\, |\, L, \, B\, \hbox{\rm have a joint eigenfunction}\, Ly=\lambda y,\, By=\mu y\}$ is included in the curve $\operatorname{Spec} {\bf C}[L,B]$ and since the latter is irreducible, they must coincide; this shows in particular that the BC polynomial is some power of an irreducible polynomial $h\colon f(\lambda,\mu )=h^{r_1}$, see Theorem~\ref{thm-wilson}.
\end{Remark}

\begin{Remark} It is clear from the form of the matrix of the extended system \eqref{eq-extended system} associated to $L-\lambda$ and $M-\mu$ that its term of highest weight is of the form $(-\lambda )^m+(-1)^{mn}\mu^n$. Let us define the semigroup of weights
\begin{gather*}
\cW=\{ an+bm\, |\, a,\, b\ \hbox{\rm nonnegative integers}\} .
\end{gather*}
In the coprime case $\gcd (n,m )=1$ (thus rank~1), by analyzing the
general solution $(a+cm)n+(b-cn)m$, it is easy to prove the following
useful statements~\cite{BC2}: (i)~every number in the closed interval
$[(m-1)(n-1),mn-1]$ belongs to $\cW$ and exactly half the numbers in the
closed interval $[1,(m-1)(n-1)]$ do not; (ii)~in this range, a
solution $(a,b)$ to $an+bm=k$ is unique.

To explain the significance of the weight, we compactify the BC curve following~\cite{Mum2} to $X= \operatorname{Proj} R$, where $R$ is the graded ring
\begin{gather*}R=\oplus_{s=0}^\infty A_s ,\qquad \text{with} \qquad A_s=\{ A\in {\mathcal A}\, |\, \ord A\leq s\}
\end{gather*}
and the operator 1 is represented by an element $e\in A_1$ (in our case the
commutative algebra~${\mathcal A}$ is ${\bf C}[L,B]$, but the construction holds
for any commutative subalgebra
of $\mathcal{D}$ that contains an
element of any sufficiently large order \cite[Remark~6.3]{SW}). That the point $P_\infty$ which we added is smooth can be seen as follows: the affine open $e\not= 0$ is $\operatorname{Spec}\big(R\big[ \frac{1}{e}\big]_0\big)= \operatorname{Spec} {\mathcal A}$ (the subscript~0 signifies the degree zero component); the affine open where $L\not= 0$ is $\operatorname{Spec}\big(R\big[\frac{1}{L}\big]_0\big)$ and the completion of this ring in the $e$-adic topology is~${\bf C}[[z]]$ if $z$ corresponds to $L^iB^j/L^k$ with $in+jm=kn-1$ (basically we are using ${\mathcal L}^{-1}$ as a local parameter, with ${\mathcal L}=L^{1/n}$). Thus, the weight is the valuation at $P_\infty$ of a function in ${\mathcal A}$, $\cW$ is the Weierstrass semigroup and the number of gaps $g=\frac{(m-1)(n-1)}{2}$ is the genus of~$X$ if there are no finite singular points.
\end{Remark}

\begin{Lemma}[\cite{BC}] If $[L,B]=0$ then there exists a~polynomial in two variables $f(\lambda,\mu )\in {\bf C}[\lambda,\mu ]$ such that $f(L,B)\equiv 0$, if we assign ``weight'' $na+mb$ to a~monomial $\lambda^a\mu^b$ where $n=\ord L$, $m=\ord B$, $\gcd(n,m)=1$, then the terms of highest weight in $f$ are $\alpha\lambda^m+\beta\mu^n$ for
some constants~$\alpha$,~$\beta$.
\end{Lemma}

The first result of this sort appeared is the $1928$ paper~\cite{BC} by Burchnall and Chaundy. More general rings were later studied in \cite{Good,K77, Richter} in the case of Ore extensions.

There are some potentially misleading features of the rank of the algebra ${\bf C}[L,M]$, but the next result settles the issue. Obviously
\begin{gather*}
 \rk(L,M) \geq \rk ( {\bf C}[L,M] ) .
\end{gather*}

\begin{Theorem}[{\cite[Appendix for a rigorous proof]{W}}]\label{thm-wilson}
Let $K$ be the field of fractions of the ring~${\bf C}[[x]]$ or ${\bf C}\{x \} $. Given $L$, $M$ commuting differential operators in~$K[\partial]$. Let $r$ be the rank of the algebra ${\bf C}[L,M]$, $f$ the BC polynomial of $L$ and $M$ in~\eqref{eq-res} and $\Gamma$ their spectral curve. The following statements hold:
\begin{enumerate}\itemsep=0pt
 \item[$(1)$] $f=h^r$, where $h$ is the unique $($up to a constant multiple$)$ irreducible polynomial satisfied by $L$ and $M$;

 \item[$(2)$] $r=\gcd\{\ord(Q)\,|\, Q\in {\bf C}[L,M]\}$;

 \item[$(3)$] $r=\dim (V(\lambda_0,\mu_0))$ where $V(\lambda_0,\mu_0)$ is the space of common solutions of $Ly=\lambda_0 y$ and $My=\mu_0 y$, for any non-singular $(\lambda_0,\mu_0)$ in $\Gamma$.
\end{enumerate}
\end{Theorem}

Observe that whenever $f$ is an irreducible polynomial then $r=1$ and otherwise the tracing index of the curve $\Gamma$ is $r>1$. Furthermore, $r$ can be computed by means of~\eqref{eq-res} and Theorem~\ref{thm-wilson}(1). It may happen that $\rk(L,M)> \rk ( {\bf C}[L,M] )$.

\begin{Definition}\label{def-truerank}Let $(K, \partial )$ be a differential field, and commuting differential operators $L$, $M$ with coefficients in~$K$. If $r=\rk(L,M)=\rk ( {\bf C}[L,M] )$, we call $L$, $M$ a {\it true rank $r$ pair} otherwise a {\it fake rank~$r$ pair}.
\end{Definition}

The first example of a true rank~$2$ pair was given by Dixmier in \cite[Proposition~5.5]{Dix}. Other families of true rank pairs were provided in \cite{Mo1, Mo2}. In~\cite{Mi1}, Mironov gave a family of operators of order~$4$ and arbitrary genus, proving the existence of their true rank~2 pairs.

We define the {\it true rank of a commutative algebra} as the rank of the maximal commutative algebra that it is contained in.

\begin{Proposition}If a commutative subalgebra of the Weyl algebra has prime rank, then it is a true-rank algebra.
\end{Proposition}
\begin{proof}Let $W$ be a commutative subalgebra of rank $r$. A larger commutative subalgebra would have rank $s$
divisor of $r$ because it would correspond to a vector bundle of rank $s$ over a curve~$\Sigma$
that covers the spectral curve $\Gamma$ of~$W$ by a map of degree~$d$, so that
$r=s\cdot d$. In our case $s=1$, and by Krichever's theorem on rational KP solutions~\cite{krichever}
they must vanish as $|x|$ approaches infinity, thus if polynomial they must be zero.
\end{proof}

Note, however, that a true-rank algebra need not be maximal-commutative.

\begin{Remark}In that context, we note two misleading features of the rank and we highlight the fact that the rank is a subtle concept:
\begin{enumerate}\itemsep=0pt
\item If $L$, $B$ are of order 2, 3 and satisfy $B^2=4L^3-g_2L-g_3$, then ${\bf C}[L,L^2+B]$ has rank 1 even though the generators have order~2,~4.
\item Note also that ${\bf C}[L]$ has rank $\ord (L)$, which shows that an algebra of rank~1 cannot be of type ${\bf C}[L]$ except for the trivial (normalized) $L=\partial$.
\item We produce fake-rank commutative subalgebras of the first Weyl algebra. Working with Dixmier's operators $L$ of order 4 and $B$ of order~6 in~\eqref{eq-Dixintro}.

\begin{itemize}\itemsep=0pt
\item We use the new pair $M=B^3$, $N=L^3$ to construct an algebra of fake
rank $6=\gcd(18,12)$.
Since $B^2=L^3+\alpha$, the equation of an elliptic curve~$E$,
$B^6$ equals a~polynomial of degree three in $N$,
$M^2=(N-a)(N-b)(N-c)$, which is again the equation of a~(singular) elliptic curve
$F$. Since ${\bf C}[M,N]\subset{\bf C}[L,B]$, there is a map
$\tau \colon E\rightarrow F$, in fact of degree three so that the direct image of a rank~2
bundle on $E$ has rank six on $F$, as expected for the common solutions of
$N-\lambda$, $M-\mu$. In fact, by the Riemann--Hurwitz
formula $2-2h=d(2-2g)-b$, where $d$ is the degree~(3) and~$b$ the total
ramification, in the elliptic case of $g=1$, $h=0$ and $b$ given by the
singular point and the point at infinity.
Therefore, the true rank of ${\bf C}[M,N]$ must also be 2.

\item For one more example of fake rank, instead we can take the square of the previous equation to obtain $B^4=L^6+2\alpha L^3+\alpha^2$, which gives an elliptic curve, and its algebra ${\bf C}\big[B^4,L^6+2\alpha L^3\big]$, which has rank~6 being the same as ${\bf C}[ B]$.
\end{itemize}

\item The (3,4) curve, cf.~\cite[Section 2 (first paragraph)]{eemop}, provides an elliptic algebra of fake rank 2: by taking $\mu_1, \mu_3, \mu_5, \mu_9=0$ we get an elliptic equation for $y$ and~$x^2$, the functions on the curve that play the role of the two commuting operators $L$ and $B$ of orders~4,~6 respectively. However, this is not a Weyl algebra because the coefficients are more general functions than polynomials.
\end{enumerate}
\end{Remark}

\section{GCD at each point of the spectral curve }\label{sec-gcd}

For a differential field $(K,\partial)$, the ring of differential operators $\mathcal{D}=K[\partial]$ admits Euclidean division. For instance in \cite{W} $K$ is the field of fractions of the ring ${\bf C}[[x]]$ or ${\bf C}\{x \} $. Given~$L$,~$M$ in $\cD$, if $\ord(M)\geq \ord(L)$ then $M=qL+r$ with $\ord(r)<\ord(L)$, $q,r\in K[\partial]$. Let us denote by $\gcd(L,M)$ the greatest common (right) divisor of $L$ and $M$.

The tool we have chosen to compute the greatest common divisor of two differential operators is the differential subresultant sequence, see \cite{Ch,Li}.
We summarize next its definition and main properties.

We introduce next the subresultant sequence for differential operators $L$ and $M$ in $K[\partial]$ of orders $n$ and $m$ respectively.
For $k=0,1,\ldots ,N:=\min\{n,m\}-1$ we define the matrix $S_k(L,M)$ to be the coefficient matrix of the extended system of differential operator
\[\Xi_k(L,M)=\big\{\partial^{m-1-k} L,\ldots, \partial L, L, \partial^{n-1-k}M,\ldots ,\partial M, M\big\}.\]
Observe that $S_k(L,M)$ is a matrix with $n+m-2k$ rows, $n+m-k$ columns and entries in $K$.
For $i=0,\dots ,k$ let $S_k^i(L,M)$ be the squared matrix of size $n+m-2k$ obtained by removing the columns of $S_k(L,M)$ indexed by $\partial^{k},\ldots ,\partial,1$, except for the column indexed by~$\partial^{i}$. Whenever there is no room for confusion we denote~$S_k(L,M)$ and $S_k^i(L,M)$ simply by~$S_k$ and $S_k^i$ respectively.
The {\it subresultant sequence} of~$L$ and~$M$ is the next sequence of differential operators in $K[\partial]$:
\begin{gather}\label{eq-subresseq}
\cL_k(L,M)=\sum_{i=0}^k \det\big(S_k^i\big) \partial^i,\qquad k=0,\ldots ,N.
\end{gather}

Given commuting differential operators $L$ and $M$ with coefficients in $K$. Let us assume that~$L$,~$M$ is a true rank~$r$ pair. The differential subresultant allows closed form expressions of the greatest common factor of order $r$ of $L-\lambda_0$ and $M-\mu_0$ over a non-singular point $(\lambda_0,\mu_0)$ of their spectral curve~$\Gamma$, defined by $f(\lambda , \mu )=0$. From the main properties of differential resultants~\cite{McW}, we know that $f(\lambda_0,\mu_0)=0$ is a condition on the coefficients of the operators $L-\lambda_0$, $M-\mu_0$ that guarantees a right common factor. Then, for any non-singular $(\lambda_0,\mu_0)$ in $\Gamma$, the nontrivial operator (found by the Euclidean algorithm) of highest order for which $M-\mu_0 =T_1G_0$, $L-\lambda_0=T_2G_0$ is $G_0=\gcd(L-\lambda_0,M-\mu_0)$.

The next theorem explains how to compute $G_0$ using differential subresultants when we consider operators in the first Weyl algebra in Section~\ref{sec-A1C}.

\begin{Theorem}\label{thm-gcd} In the previous notations, consider commuting differential operators $L$ and $M$ with coefficients in ${\bf C}(x)$. Assume $L$, $M$ is a true rank~$r$ pair, then for any non-singular $(\lambda_0,\mu_0)$ in $\Gamma$ the greatest common divisor $G_0$ of $L-\lambda_0$ and $M-\mu_0$ is the order $r$ differential operator
\begin{gather}\label{eq-gcd}
 G_0=\gcd(L-\lambda_0,M-\mu_0)=\cL_r(L-\lambda_0,M-\mu_0).
\end{gather}
Furthermore, the subresultants $\cL_n(L-\lambda_0,M-\mu_0)$ are identically zero for $n=0,\ldots ,r-1$.
\end{Theorem}
\begin{proof}Recall that, the $\gcd(L-\lambda_0, M-\mu_0)$ is nontrivial (it is not in ${\bf C}(x)$) if and only if $f(\lambda_0,\mu_0)=\dres(L-\lambda_0, M-\mu_0)=0$, because of \cite{Prev} and \cite[Theorem 4]{Ch}. Furthermore, from Theorem~\ref{thm-wilson} and Theorem~4 from~\cite{Ch}, if the pair $L$, $M$ is true rank~$r$, then the greatest common divisor of $L-\lambda_0$, $M-\mu_0$ can be computed using the $r$th subresultant, for any non-singular point $(\lambda_0,\mu_0)$ in $\Gamma$. Summarizing, we obtain the result.
\end{proof}

By this theorem, we obtain an {\it explicit} presentation of the right factor of order $r$ of $L-\lambda_0$ and $M-\mu_0$ that can be {\it effectively computed}. Hence an {\it explicit} description of the fiber $\cF_{P_0}$ of the rank~$r$ spectral sheave $\cF$ in the terminology of \cite{BZ,PW}, where the operators are given in the ring of differential operators with coefficients in the formal power series ring~${\bf C}[[x]]$.

The next example illustrates the computation of greatest common divisors using differential subrestultants for a pair of true rank~$2$ operators over a spectral curve of genus~$2$.

\begin{Example}\label{ex1}Using a Gr\"unbaum's style approach~\cite{Grun}, we search for operators of order~$4$ in~$A_1({\bf C})$ that commute with a nontrivial operator (not in~${\bf C}[L_4]$) of order~$10$. We fix
\begin{gather}\label{operator-L4general}
 L_4=\big(\partial^2+x^4+1\big)^2+U(x)\partial+W(x),
\end{gather}
where $U(x)=u_3x^3+u_2 x^2+u_1 x +u_0$ and $W(x)=w_2x^2+w_1 x +w_0 $ in ${\bf C}[x]$. Forcing the commutator $[L_4,M_{10}]=0$, for an arbitrary operator $M_{10}$ of order $10$, we obtain that the only nontrivial answers are:
\begin{enumerate}\itemsep=0pt
\item $U(x)=0$ and $W(x)=4x^2+w_0$ or $W(x)=8x^2+w_0$, which are self-adjoint examples given in~\cite{Og2016}, with $g=1$ and $g=2$ respectively.\label{ex-1-1}

\item $U(x)=\pm 4{\rm i}$ and $W(x)=4x^2+w_0$, which is a non self-adjoint case, with $g=1$, see~\cite{Grun,MZ2014}.\label{ex-1-2}

\item $U(x)=\pm 8{\rm i}$ and $W(x)=16x^2+w_0$, which is a non self-adjoint case with $g=2$, as we will prove in Section~\ref{sec-contruction}, Example~\ref{ex3}.\label{ex-1-3}

\item $U(x)=\pm 12{\rm i}$ and $W(x)=12x^2+w_0$, which is a non self-adjoint case, with $g=2$ as we will prove in Section~\ref{sec-contruction}, Example~\ref{ex3}.\label{ex-1-4}
\end{enumerate}

To illustrate the computation of the greatest common divisor using differential subresultants let us consider the differential operator
\begin{gather}\label{eq-nonselfadj1}
 L_{4}= \big( \partial^2 + x^4+1 \big)^2+8{\rm i}\partial+16x^2.
\end{gather}
From a family of operators of order $10$ commuting with $L_4$ we choose
\begin{gather}
B_{10}=\partial^{10}+5\big(x^4+1\big)\partial^8+20\big(4x^3+{\rm i}\big)\partial^7+10\big(x^8+2x^4+64x^2+1\big)\partial^6\nonumber\\
\hphantom{B_{10}=}{} +T_5\partial^5+T_4 \partial^4 +T_3 \partial^3 + T_2 \partial^2 + T_1 \partial + T_0,\label{eq-B10ex1}
\end{gather}
for some $T_i\in {\bf C}[x]$ (not included due to their length). Moreover, the differential resultant $\dres(L_4-\lambda,B_{10}-\mu)=h(\lambda,\mu)^2$ with
\begin{gather}\label{eq-hB10}
h(\lambda,\mu) ={\mu}^{2}+R_5(\lambda)=\mu^2+\lambda\big({-}\lambda^4-56\lambda^2+288\lambda-1296\big).
\end{gather}
Thus, by Theorem \ref{thm-wilson}, $L_4$, $B_{10}$ is a true rank~2 pair that verifies $(B_{10})^2=R_5 (L_{4})$.

By Theorem \ref{thm-gcd}, for any $P_0=(\lambda_0,\mu_0)$ in the spectral curve $\Gamma$ defined by \eqref{eq-hB10}, the greatest common divisor of $L_4-\lambda_0$ and $B_{10}-\mu_0$ is given by the second subresultant $\cL_2(L_4-\lambda_0,B_{10}-\mu_0)$, see \eqref{eq-subresseq}. In fact the subresultants $\cL_n(L_4-\lambda_0,B_{10}-\mu_0)$, $n=0,1$ are zero. For details,
 \begin{gather*}
\cL_0(L_4-\lambda_0,B_{10}-\mu_0)=h(\lambda_0,\mu_0)^2=0,\\
\cL_1(L_4-\lambda_0,B_{10}-\mu_0)=\phi_1+\phi_2 \partial=0
 \end{gather*}
 with
 \begin{gather*}
\phi_2=\det \big(S_1^1\big)=4 {\rm i}\big( 18 {x}^{2}+\lambda_0 \big) h(\lambda_0,\mu_0)=0,\\
\phi_1=\det \big(S_1^0\big)= -\big( 8 \lambda_0 {x}^{2}+72 {x}^{4}+36+{\lambda_{0}}^{2}+72 {\rm i}x \big)h(\lambda_0,\mu_0)=0.
 \end{gather*}
 The greatest common divisor of $L_4-\lambda_0$ and $B_{10}-\mu_0$ equals
 \begin{gather}\label{eq-subresL4B6}
 \cL_2(L_4-\lambda_0,B_{10}-\mu_0)=\det\big(S_2^2 \big)\partial^2+ \det\big(S_2^1\big)\partial+\det\big(S_2^0\big)
 \end{gather}
 with
 \begin{gather*}
 \det\big(S_2^2\big)=576 \lambda_{0} {x}^{6}+192 {\lambda_{0}}^{2}{x}^{4}+16 {\lambda_{0}}^{3}{x}^{2}+{\lambda_{0}}^{4}+56 {\lambda_{0}}^{2}-288 \lambda_{0}+1296,\\
 \det\big(S_2^1\big)= 4 \big( {-}24 \lambda_{0} {x}^{3}-4 {\lambda_{0}}^{2}x+{\rm i}\mu_{0} \big) \big( 18 {x}^{2}+\lambda_{0} \big),\\
 \det\big(S_2^0\big)=1296+5184 {\rm i} x+1296 {x}^{4}+ \big( 56+288 {\rm i}{x}^{5}+80 {\rm i}x+192 {x}^{8}+248 {x}^{4}+288 {x}^{2} \big) {\lambda_{0}}^{2}\\
 \hphantom{\det\big(S_2^0\big)=}{} + \big( {-}288+576 {\rm i}{x}^{7}-864 {\rm i}{x}^{3}-1152 {\rm i}x+576 {x}^{10}+576 {x}^{6}+1440 {x}^{4} \big) \lambda_{0}\\
\hphantom{\det\big(S_2^0\big)=}{} + \big( {-}36-72 {\rm i}x-72 {x}^{4} \big) \mu_{0}
 + \big( {x}^{4}+1 \big){\lambda_{0}}^{4}-8 \lambda_{0} \mu_{0} {x}^{2}-{\lambda_{0}}^{2}\mu_{0}\\
\hphantom{\det\big(S_2^0\big)=}{} + \big( 8 {\rm i}{x}^{3}+16 {x}^{2}+16 {x}^{6} \big) {\lambda_{0}}^{3}.
 \end{gather*}
 Observe that $\cL_2(L_4-\lambda_0,B_{10}-\mu_0)$ is an order $2$ differential operator in $A_1({\bf C})$ and also that the monic greatest common divisor is $\partial^2-\chi_1\partial-\chi_0$ with
 \[\chi_1=-\frac{\det\big(S_2^1\big)}{\det\big(S_2^2\big)} ,\qquad \chi_0=-\frac{\det\big(S_2^0\big)}{\det\big(S_2^2\big)}.\]
 Therefore the fiber $\cF_{P_0}$ at $P_0$ of the rank $r=2$ spectral sheave $\cF$ over the curve $\Gamma$ is the order two operator $\partial^2-\chi_1\partial-\chi_0$ in total agreement with~\cite{BZ}.
\end{Example}

\begin{Remark}We would like to point out that the operators $L_4$ in Cases~\ref{ex-1-3} and~\ref{ex-1-4} of Example~\ref{ex1} are not self-adjoint and their spectral curves have genus $g=2$. We believe they are new examples of rank $2$ fourth order non self-adjoint operators with nontrivial centralizers.

The factorization of ordinary differential operators using differential subresultants, for non self-adjoint operators, is an important contribution of this work. The determinantal formulas obtained by means of~\eqref{eq-gcd} can be {\it effectively} computed. See for instance~\eqref{eq-subresL4B6}, for which we have used Maple 18 to give the final formulas.
\end{Remark}

\section{Centralizers and BC pairs}\label{sec-Dixmier}

In this section, we review a theorem by Goodearl~\cite{Good} on the description of a basis of the centralizer $\cC_{\cD}(L)$ as a free ${\bf C}[L]$-module and give the notion of BC pair.

Given commuting differential operators $L$ and $M$ in $\cD$, we observe that
\[{\bf C}[L,M]\subseteq \cC_{\cD}(L),\]
but they can be different. Since $\cC_{\cD}(L)$ is a maximal subalgebra by Corollary \ref{cor-maxcom}, we wonder when is ${\bf C}[L,M]$ a maximal subalgebra and therefore equal to the centralizer. The next result about the description of the centralizer will allow us to reach some conclusions.

The following theorem was proved in \cite{Good} in as wide a context as reasonable (more general rings of differential operators $\cD$). For instance, the ring $\cC^{\infty}$, of infinitely many times differentiable complex valued functions on the real line, is not a~field but by \cite[Corollary 4.4]{Good}, the centralizer $\cC_{\cC^{\infty}}(P)$, $P=a_n \partial^n+\cdots +a_1\partial+a_0$ is commutative if and only if there is no nonempty open interval on the real line on which the functions $\partial(a_0),a_1,\ldots ,a_n$ all vanish. Details of the evolution of the next result from various previous works are given in \cite{Good}.

\begin{Theorem}[{\cite[Theorem 1.2]{Good}}]\label{thm-good} Let $L$ be an operator of order $n$ in $\cD=K[\partial]$. Let $X$ be the set of those $i$ in $\{0,1,2,\ldots ,n-1\}$ for which $\cC_{\cD}(L)$ contains an operator of order congruent to $i$ module $n$. For each $i\in X$ choose $Q_i$ such that $\ord ( Q_{i} )\equiv i$ $(\mod \, n)$ and $Q_i$ has minimal order for this property $($in particular $0\in X$, and $Q_0=1)$. Then $\cC_{\cD}(L)$ is a free ${\bf C}[L]$-module with basis $\{Q_i\,|\, i\in X\}$. Moreover, the cardinal $t$ of a basis of $\cC_{\cD}(L)$ as a free ${\bf C}[L]$-module is a~divisor of~$n$.
\end{Theorem}

The cardinal $t$ of a basis of $\cC_{\cD}(L)$ as a free ${\bf C}[L]$-module is known as the rank of the module. We will not use this terminology to avoid confusion with the notion of rank of a set of differential operators that is being analyzed in this paper.

\begin{Remark}If the cardinal of a basis of $\cC_{\cD}(L)$ as a free ${\bf C}[L]$-module is $t=2$ then it is a free ${\bf C}[L]$-module with basis $\{1,B\}$, that is
\[\cC_{\cD}(L)=\{p_0(L)+p_1(L)B\,|\, p_0,p_1\in {\bf C}[L]\}={\bf C}[L,B].\]
\end{Remark}

The question we will try to answer, in some cases, in this paper is: Given a commutative true rank~$r$ pair $L$,~$M$, is $L$,~$M$ a basis of $\cC_{\cD}(L)$ as a free ${\bf C}[L]$-module? In the affirmative case then
\begin{gather*}
 \rk(L,M) = \rk ( {\bf C}[L,M] ) = \rk(\cC_{\cD} (L))=r.
\end{gather*}

\begin{Definition}\label{def-BC} Let $L$ be an irreducible operator in $\cD$. Given a pair $L$,~$M$ of differential operators in $\cD$, with $M\notin {\bf C}[L]$,
we will call $L$, $M$ a {\it Burchnall--Chaundy $($BC$)$ pair} if ${\bf C}[L,M]=\mathcal{C}_{\cD}(L)$.
\end{Definition}

\begin{Theorem}\label{prop-BCtrue}Let $L$,~$M$ be a commutative pair of rank $r\geq 1$ in $\cD$, with $M\notin {\bf C}[L]$. If $L$,~$M$ is a BC pair then $L$, $M$ is a true rank~$r$ pair.
\end{Theorem}
\begin{proof}Let $n$, $m$ be $ \ord (L)$ and $\ord (M)$ respectively. Since $L$,~$M$ is a BC pair and a rank $r$ pair, we have $\cC (L) = {\bf C} [L,M]$ and $r=\gcd (\ord (L), \ord (M))$. Next we will proof that the algebra ${\bf C} [L,M] $ is a~rank~$r$ algebra.

Let $s$ be the rank of ${\bf C} [L,M] $. Then $s|r$. There exists $Q\in \cC (L)$ with $s=\gcd (\ord (L), \ord (Q))$. Observe that $s<n$ and $r<n$. But, by Theorem~\ref{thm-good} we have $X=\{0, r\}$, where $X$ is the set of those $i$ in $\{0,1,2,\ldots ,n-1\}$ for which $\cC_{\cD}(L)$ contains an operator of order congruent to $i$ module $n$. Hence $s=r$, and the pair $L$,~$M$ is a true rank~$r$ pair.
\end{proof}

Observe that the converse of Theorem \ref{prop-BCtrue} is not true in general. See examples in Section \ref{sec-contruction}.

\begin{Remark}
The ring $\mathbf{C}[L,B]$ is a priori only a subring of the affine ring of
the spectral curve, as is clear from Remark \ref{locpar}.
This is a crucial problem, around which we built our algorithm
{\tt BC pair}, as stated in the Introduction.
Using the parameter~$k$,
Segal and Wilson give an illustration of what can be viewed as a containment
of commutative subalgebras, and the surjective morphisms between the attendant
spectral curves~\cite[Section~6]{SW}. In particular, if
$\mathcal{C}_\mathcal{D}(L)=\mathbf{C}[L,B]$, the spectral curve is special,
in that it can be embedded in the plane with only one smooth point at
infinity; the noted Klein quartic curve gives a non-example of such a~curve~\cite{KMP}. Of course, in the case of a hyperelliptic curve defined by~$B^2$ equalling a polynomial in~$L$, the ring of the affine curve is indeed
$\mathbf{C}[L,B]$, unless the curve has singular points and
in that case the ring of the desingularization is larger; examples of this can be constructed by transference, but in order to stay in the Weyl algebra, one has to ensure that after conjugation the ring still has polynomial coefficients.
\end{Remark}

\section[Gradings in $A_1 ({\bf C})$ and the Dixmier test]{Gradings in $\boldsymbol{A_1 ({\bf C})}$ and the Dixmier test} \label{subsect-grading}

In the remaining parts of this paper we will consider differential operators in the first Weyl algebra $A_1({\bf C})$. In this section we define an appropriate filtration of $A_1({\bf C})$ to use a lemma by Dixmier~\cite{Dix} that we call the \texttt{Dixmier test}.

Next, we present some well known techniques for grading the first Weyl algebra $A_1 ({\bf C})$, for a field of zero characteristic ${\bf C}$, see for instance \cite{BM,CN}. For non zero $P \in A_1 ({\bf C})$, say $P= \sum_{i , j }a_{i j} x^i \partial^j$, we denote by $\mathcal{N}(P)$ its {\it Newton diagram}
$\mathcal{N}(P)=
 \big\{ (i ,j )\in \mathbb{N}^2 \, | \, a_{i j} \not=0 \big\}$.
Given non negative integers $p$, $q$ such that $p +q >0$, we consider the linear form
\begin{gather*}
 \Lambda_{p,q} (i,j)= p i+qj.
\end{gather*}

\begin{Lemma}[see \cite{CN}]With the previous notation, the function
\begin{gather*}
 \delta \colon \ A_1 ({\bf C}) \rightarrow \mathbb{Z} \cup \{ -\infty \} , \qquad \delta (P ) =\max \{ \Lambda_{p,q} (i,j) \, | \, (i,j)\in \mathcal{N} (P) \}
\end{gather*}
is an admissible order function on $A_1 ({\bf C})$. Moreover, the family of ${\bf C}$-vector spaces
\begin{gather*}
 G_\delta^s = \{ P \in A_1 ({\bf C}) \, | \, \delta (P) \leq s \} , \qquad s\in \mathbb{Z} ,
\end{gather*}
is an increasing exhaustive separated filtration of
$A_1 ({\bf C})$, and it is called the $\delta_{p,q}$-filtration of~$A_1 ({\bf C})$ $($associated to the linear form $\Lambda_{p,q})$.
\end{Lemma}

Let us consider the commutative ring of polynomials ${\bf C}[\chi , \xi ]$ and the ${\bf C}$-algebra isomorphism:
\begin{gather*}
\phi \colon \ {\bf C}[\chi ,\xi] \rightarrow \operatorname{gr}_\delta (A_1 ({\bf C})) , \qquad
\phi(\chi )=\sigma (x) , \qquad \phi(\xi )=\sigma (\partial ),
\end{gather*}
where $\sigma (P)$ is the principal symbol of the operator $P$ with respect to the $\delta_{p,q}$-filtration. Moreover~$\phi$ is an isomorphism of graded rings where the degree function in ${\bf C}[\chi ,\xi]$ is given by the linear form $\Lambda_{p,q}$, that is $\deg \big(\chi^i \xi^j \big)= \Lambda_{p,q} (i,j) = p i+qj$. Moreover
\begin{gather}\label{eq-symbol}
\sigma (LM) =\sigma (L) \sigma (M).
\end{gather}

Let $P$ be an operator with $m=\delta (P)$. We call {\it the initial part of the operator} $P$ the homogeneous operator:
\begin{gather*}
\Ini(P)= \sum_{\Lambda(i , j) =m}a_{i j} x^i \partial^j.
\end{gather*}

\begin{Remark}From now on we identify $\sigma (P)$ and $\phi^{-1} (\sigma (P) )$ for each operator $P$.
\end{Remark}

For the convenience of the reader we recall a result from Dixmier work \cite{Dix} that will be useful in the next sections. The next result is \cite[Lemma 2.7]{Dix}, using the previous terminology. We will call this result the \texttt{Dixmier test}.

\begin{Lemma}[\texttt{Dixmier test}]\label{lema-Dix-2.7}
With the previous notation, let us consider the $\delta_{p,q}$-filtration of~$A_1 ({\bf C})$.
Given $ L $ and $ M $ two non-zero operators in $ A_1 ({\bf C})$, with $ v = \delta (L) $ and $w = \delta (M)$. The following statements hold:
\begin{enumerate}\itemsep=0pt
 \item[$1.$] There is a unique pair $ T$, $U $ of elements of $ A_1 ({\bf C}) $ with the following properties:
 \begin{enumerate}\itemsep=0pt
 \item[$(a)$] $[L, M]= T+U$;
 \item[$(b)$] $\mathcal{N}(T)=\mathcal{N}(\Ini (T))$ and $\delta (T)= v+w-(p+q)$;
 \item[$(c)$] $\delta (U)< v+w-(p+q)$.
 \end{enumerate}
 \item[$2.$] The following conditions are equivalent:
 \begin{enumerate}\itemsep=0pt
 \item[$(a)$] $T=0$;
 \item[$(b)$] $ \frac{\partial \sigma (L)}{\partial \chi } \frac{\partial \sigma (M)}{\partial \xi } - \frac{\partial \sigma (L)}{\partial \xi } \frac{\partial \sigma (M)}{\partial \chi } =0 $;
 \item[$(c)$] $\sigma (M)^v = c \sigma (L)^w$, for some constant $c$.
 \end{enumerate}
 \item[$3.$] If $T\not=0$, then the symbol of $[L, M]$ is
 $\sigma ( [L , M] ) = \frac{\partial \sigma (L)}{\partial \chi } \frac{\partial \sigma (M)}{\partial \xi } - \frac{\partial \sigma (L)}{\partial \xi } \frac{\partial \sigma (M)}{\partial \chi }$.
\end{enumerate}
\end{Lemma}

By means of Lemma \ref{lema-Dix-2.7}(2c), we can decide on the divisors of the orders of the operators of the centralizer of a given differential operator~$L$.

\begin{Lemma}\label{lem-filtration} Let $L\not=\partial^n$ be an order $n$ operator in normal form in $A_1 (\mathbf{C})$. There exists a~unique linear form $\Lambda_{p,q} (i,j) =pi+qj$ with $p$, $q$ non negative integers, $p+q>0$, such that $\delta_{p,q}(L) = \Lambda_{p,q}(0,n) = \Lambda_{p,q}(a,b)$ for some $(a,b)\in \cN (L)\setminus \{(0,n)\}$.
 \end{Lemma}

We will call the $\delta_{p,q}$-filtration associated to the linear form defined in Lemma~\ref{lem-filtration}, the {\it test-filtration} for~$L$.

\begin{Corollary}\label{corolario-Dix1} Let $L$ be an order $n$ operator in normal form in $A_1 (\mathbf{C})$. Let us consider the test-filtration for~$L$ in~$A_1 ({\bf C})$.
We will assume that $ \phi^{-1}(\sigma (L)) $ is a power of an irreducible polynomial~$g \in {\bf C} [\chi , \xi ]$. Given~$M$ in the centralizer $\cC (L)$ then $ \phi^{-1}(\sigma (M)) $ is also a power of~$g$.
\end{Corollary}

\begin{Corollary}\label{corolario-Dix2} Given $ L $ and $ M $ two non-zero operators in $ A_1 ( {\bf C}) $. Assume $ \phi^{-1}(\sigma (L)) = \big(\xi^p + \chi^2\big)^2$ for some positive integer $p$. If~$M$ is in the centralizer $\cC (L)$, then $\ord(M) $ is congruent with $0$ or $p$ modulo $2p$.
\end{Corollary}
\begin{proof}Take $\Lambda(i,j)=pi+2j $ and consider the $\delta_{p,2}$-filtration of $A_1 ({\bf C})$. Then, by Corollary~\ref{corolario-Dix1}, the order of $M$ is $\ord(M)=pb$ for some non negative integer $b$. But, $b=2s+\epsilon$ with $\epsilon =0$ or~$1$. Then the result follows.
\end{proof}

\begin{Example}Let us consider $L_{2p} = \big(\partial^p+x^2+\alpha\big)^2+2\partial$ for some positive integer~$p$. Take $\Lambda(i,j)=pi+2j $ and consider the $\delta_{p,2}$-filtration of $A_1 ({\bf C})$. By Corollary~\ref{corolario-Dix2}, for any monic operator $M$ in the centralizer $ \cC (L_{2p} )$, we have
\begin{gather*}
 \ord{(M)}=0 \ \ \mod (2p) \qquad \text{or} \qquad \ord{(M)}=p \ \ \mod (2p).
\end{gather*}
By Theorem~\ref{thm-good} if the centralizer is nontrivial, it equals $\cC(L_{2p})={\bf C}[L_{2p},X_p]$ with $X_p$ the operator of minimal order $p(2s+1)$, $s\neq 0$, in $\cC(L_{2p})$.
Observe that for $p=3$ this is the Fourier transform of Dixmier's example \eqref{eq-Dixintro}~\cite{Dix}. In this case by Theorem~\ref{thm-good} the centralizer is nontrivial and $X_3$ has order $9$. The pair $L$, $B=X_3$ is true rank~$3$.
\end{Example}

\section[Order 4 operators in $A_1({\bf C})$]{Order 4 operators in $\boldsymbol{A_1({\bf C})}$}\label{sec-A1C}

In this section we apply the previous results to operators of order $4$ in $A_1({\bf C})$. We will prove that for any operator of order~$4$, if non trivial, its centralizer is the ring of a plane curve (see Corollary \ref{thm-centralizerL4} and important consequences in Proposition~\ref{thm-B}).

First, recall that, as in Gr\"unbaum's work~\cite{Grun}, a general fourth order differential operator in~$K[\partial]$ can be given by
\begin{gather}\label{op-grun}
 \left(
 \partial^2 +\dfrac{c_2 }{2}
 \right)^2
 +2c_1 \partial +c_1' +c_0 , \qquad \text{with}\quad c_0 , c_1 , c_2 \in K ,
\end{gather}
after a Liouville transformation. For this reason, in this section we will consider operators of order $4$ in $A_1 ({\bf C})$ of the form
\begin{gather}\label{op-weyl}
 L_4= \big( \partial^2 +V(x) \big)^2 +U(x)\partial +W(x) ,
\end{gather}
with $U(x)$, $V(x)$ and $W(x)$ polynomials in ${\bf C}[x]$.

\begin{Remark} In \cite{Grun} it is proved that equation~\eqref{op-grun} with $c_1 \equiv 0$ is the self-adjoint case. Moreover, A. Mironov (see \cite{Mi1}) considered the self-adjoint case in the first Weyl algebra, that is $U\equiv 0$ in~\eqref{op-weyl}. He proved the Novikov's conjecture: the existence of $M$ in $\cC(L_4)$ such that $h(L_4,M)=0$ for $h(\lambda,\mu)=\mu^2+R_{2g+1}(\lambda)$ the defining polynomial of a genus~$g$ curve~$\Gamma$; furthermore this operator~$L_4$ has an order $2$ factor at each point of~$\Gamma$.
\end{Remark}

\subsection{Centralizers}\label{sec-cen4}

Our goal is to prove that the centralizer $\cC(L_4)$ of $L_4$ in $\cD=A_1({\bf C})$ is either equal to ${\bf C}[L_4]$ or to ${\bf C}[L_4,B]$, for an operator $B$ of order $4k+2$ such that $L_4$, $B$ is a true rank~$2$ pair. To avoid trivial cases, we assume~$L_4$ to be irreducible in~$\cD$. For instance if $L_4=\big(\partial^2+V(x)\big)^2$ and $B=\partial^2+V(x)$ then $\cC (L_4) =\cC(B)$ is a rank $1$ algebra.

\begin{Theorem}\label{thm-order}Let $L_4$ be an irreducible operator of order $4$ in $A_1 ({\bf C})$ as in~\eqref{op-weyl}. Assume that $\deg(V) > \max \big\{ \frac{1}{2} \deg (U) , \frac{1}{2} \deg (W) \big\}$. Then any~$M$ commuting with~$L_4$ has even order.
\end{Theorem}
\begin{proof}Let $p=\deg(V)$ be an odd integer. Let us consider the $\delta_{2,p}$-filtration of $A
_1 ({\bf C})$, with $\Lambda_{2,p} (i,j)=2i+pj$. By~\eqref{eq-symbol}, we have $\sigma (L_4)=\big[\sigma \big(\partial^2 +V\big)\big] ^2 = \big(\xi^2 +\chi^p \big)^2$. But, by \texttt{Dixmier test} (Corollary~\ref{corolario-Dix1}), $\sigma (M ) ^{4p}=\sigma (L_4 )^{pm} =\big(\xi^2 + \chi^p \big)^{2pm}$, with $m=\ord (M)$. Therefore, since $ \xi^2 + \chi^p $ is irreducible in $k [\chi , \xi ] $, $\sigma (M ) =\big(\xi^2 + \chi^p \big)^q $ for some $q$. Then $4pq=2pm$. Thus, $M$ has even order in this case.
\medskip
Next assume $p=2s=\deg(V)$, an even integer. Now, we have
\begin{gather*}
\sigma (M ) ^{4p}=\sigma (L_4 )^{pm} =\big(\xi^2 + \chi^p \big)^{2pm}= \big(\xi + i\chi^s \big)^{2pm}\big(\xi - i\chi^s \big)^{2pm}, \qquad \text{with} \quad m=\ord (M).
\end{gather*}
Then $\sigma (M )= \big(\xi + i\chi^s \big)^{a}\big(\xi - i\chi^s \big)^{b}$, because $k [\chi , \xi ] $ is an unique factorization domain. Hence, comparing multiplicities, we have $4pa=2pm$ and $4pb=2pm$. So, $M$ has even order, as was stated in the theorem.
\end{proof}

\begin{Remark}The previous result was proved in~\cite{DS} for the case $V(x)=\alpha_3 x^3+\alpha_2 x^2+\alpha_1 x+\alpha_0$, $U(x)=0$ and $W(x)=\alpha_3 g (g+1)$, with $\alpha_3\neq 0$, using different methods than those described in this work.
\end{Remark}

\begin{Lemma}\label{lemma-Miro} Let $L_4$ be an irreducible operator of order $4$ in $A_1({\bf C})$ as in~\eqref{op-weyl}. If $\cC(L_4 )\not= \mathbf{C}[L_4 ] $ then $\deg(V) > \max \big\{ \frac{1}{2} \deg (U) , \frac{1}{2} \deg (W) \big\}$.
\end{Lemma}
\begin{proof}Let us consider the $\delta_{2,p}$-filtration, with $p=\deg (V)$. Observe that if $\deg (V)\leq \frac{1}{2} \deg (U) \allowbreak =\frac{u}{2}$, the leading form of $L_4$ is $\partial^4+c_1 x^u \partial$; or if $\deg (V) \leq \frac{1}{2}\deg (W) =\frac{w}{2}$, this leading form is $\partial^4+c_2 x^w$. In neither case its leading form is the square of another form of lower degree, thus the centralizer is trivial. Consequently the statement follows, because of Dixmier's lemma~\ref{lema-Dix-2.7}.
\end{proof}

 We recall that by Theorem \ref{thm-good} the centralizer of and operator $L_4$ is the free ${\bf C}[L_4]$-module with basis $X=\{X_j\,|\, j\in J\}$, being $J$ the subset of $I=\{0,1,2,3\}$ of those $j\in I$ for which there exists an operator $X_j\in \mathcal{C}(L_4)$ of minimal order congruent with $j$ mod $4$. Therefore, we can establish the following claim.

\begin{Corollary}\label{thm-centralizerL4}
Let $L_4$ be an irreducible operator of order $4$ in $A_1({\bf C})$ as in \eqref{op-weyl}, {such that $\cC(L_4 )\not= \mathbf{C}[L_4 ]$}.
Then
\begin{gather*}\mathcal{C}(L_4)={\bf C}[L_4]\langle 1, X_2 \rangle={\bf C}[L,X_2]\end{gather*}
for an operator $X_2$ of minimal order $2(2g+1)$, for $g\neq 0$, that is~$\mathcal{C}(L_4)$ equals the free ${\bf C}[L_4]$-module with basis $\{1,X_2\}$. Furthermore the pair $L_4$, $X_2$ is BC and true rank~$2$.
\end{Corollary}\label{prop-truerank}

\begin{proof}By Lemma \ref{lemma-Miro}, Theorem \ref{thm-good}, and Theorem \ref{thm-order} and the hypothesis, the centralizer of~$L_4$ is the free ${\bf C}[L_4]$-module with basis $\{1,X_2\}$, in notations of Theorem~\ref{thm-good}, that is
\begin{gather*}{\bf C}[L_4]\langle 1, X_2\rangle=\{p_0(L_4)+p_1(L_4)X_2\,|\, p_0,p_1\in {\bf C}[L_4]\}.\end{gather*}
By \eqref{eq-freemod}, it equals ${\bf C}[L_4,X_2]$.
The pair $L_4$, $X_2$ satisfies Definition~\ref{def-BC} and Theorem~\ref{prop-BCtrue} implies it is true rank~$2$.
\end{proof}

\begin{Remark}The previous corollary is only the first example of how to apply \texttt{Dixmier test} to prove results on the structure of the basis of the centralizer of an operator of the first Weyl algebra. We believe that similar results can be obtained for higher order operators.
\end{Remark}

By Theorem \ref{thm-wilson}, given a true rank~$2$ pair $L_4$, $M$ in $A_1 ({\bf C})$ , the spectral curve $\Gamma$ is defined by a polynomial $h$ in ${\bf C}[\lambda,\mu]$ that verifies
\begin{gather}\label{eq-h2}
 f=\dres(L_4-\lambda,M-\mu)=h^2.
\end{gather}
In addition $\Gamma$ is a hyperelliptic curve defined by an equation $\mu^2=b_0(\lambda)+b_1(\lambda)\mu$ with $b_0(\lambda),b_1(\lambda)\allowbreak \in {\bf C}[\lambda]$. Thus $M^2=b_0(L_4)+b_1(L_4)M$ and
\begin{gather}\label{eq-freemod}
{\bf C}[L_4,M]=\left\{\sum \alpha_{i,j}L_4^iM^j\,|\, \alpha_{i,j}\in {\bf C} \right\}=\{p_0(L_4)+p_1(L_4)M\,|\, p_0,p_1\in {\bf C}[L_4]\}.
\end{gather}

\begin{Remark}\label{rem-trivial}
Assume $\cC (L_4) = {\bf C } [ L_4 , X_2 ]\neq {\bf C } [ L_4]$, for an operator $X_2$ of minimal order $2(2g+1)$, for $g\neq 0$.
\begin{enumerate}\itemsep=0pt
 \item {Observe that if $M=p_0(L_4)+p_1(L_4)X_2$ has order $4q$, $q>0$ then it means that}
 \begin{gather*}\ord(p_0(L_4))\geq \ord (p_1(L_4))X_2.\end{gather*}
Note that a nonzero $M_1=M-p_0(L_4)$ has order $4q+2$, for some $q>0$.

\item In particular, we can detect if $M=p_0(L_4)$ by means of the differential resultant. In fact, by the Poison formula for the differential resultant (see \cite{Ch}) then $\dres(L_4-\lambda,M-\mu)$ equals $(p_0(\lambda)-\mu)^4$. Obviously, in this case ${\bf C}[L_4,M]={\bf C}[L_4]$.

\item If $\ord(M)=\ord(X_2)$ then $M-X_2\in {\bf C}[L_4]$ and ${\bf C}[L_4,M]={\bf C}[L_4,X_2]$. Otherwise, if $\ord(M)>\ord(X_2)$ then ${\bf C}[L_4,M]\subset {\bf C}[L_4,X_2]$, the equality cannot hold.
\end{enumerate}
\end{Remark}

The next result contains essential claims to establish an algorithm.

\begin{Proposition}\label{thm-B}
Let $L_4$ be an irreducible operator of order $4$ in $A_1 ({\bf C})$ as in \eqref{op-weyl}. Assume $\cC (L_4) = {\bf C } [ L_4 , X_2 ]\neq {\bf C } [ L_4]$, for an operator $X_2$ of minimal order $2(2g+1)$, for $g\neq 0$.
Given $M=p_0 (L_4) +p_1 (L_4) X_2$ in $\mathcal{C}(L_4)$ with $p_1 \not=0$, then:
\begin{enumerate}\itemsep=0pt
 \item[$1.$] There exists an operator $B_g$ in $\cC(L_4)$ such that ${\bf C } [ L_4 , X_2 ]={\bf C } [ L_4 , B_g ]$ and the spectral curve associated to the pair $L_4$, $B_g$ is a hyperellipctic curve defined by a polynomial $h(\lambda,\mu)=\mu^2-R_{2g+1}(\lambda)$, with $R_{2g+1}(\lambda)\in {\bf C}[\lambda ]$ of degree $2g+1$.

 \item[$2.$] $\dres(L_4-\lambda,M-\mu)=\big(\mu^2-b_1(\lambda) \mu-b_0(\lambda)\big)^2$, with $b_0,b_1\in {\bf C}[\lambda]$ and $p_0(L_4)=b_1(L_4)/2$. \label{lem_remove_p0}

 \item[$3.$] $M_1=M-p_0(L_4)$, has order $2(2q+1)$, with $p_1\in {\bf C}[\lambda]$ of degree $4(q-g)$ and it verifies $M_1^2=R_{2q+1}(L_4)$, for $R_{2q+1}(\lambda)=p_1(\lambda)R_{2g+1}(\lambda)$.
\end{enumerate}
\end{Proposition}

\begin{proof}1.~We know that $X_2^2=b_0( L_4)+b_2(L_4)X_2$. We easily check that $B=X_2-(1/2)b_1(L_4)$ verifies
$B^2 = R_a (L_4)$, for $R_a(\lambda)\in {\bf C}[\lambda]$ of degree~$a$. Since $\mathcal{C}(L_4)={\bf C}[L_4, B]$ it remains to prove that $a=2g+1$. Let us consider the $\delta_{2,p}$-filtration of $A_1 ({\bf C})$, with $p=\deg(V)$. Taking symbols in $B^2 = R_a (L_4)$, we have
\begin{gather*}
\sigma (B)^2 =\sigma (L_4)^a = \big(\xi^2 + \chi^p \big)^{2a} .
\end{gather*}
Then $2 (2g+1) =2a$. Finally $a=2g+1$.

2.~We know that \eqref{eq-h2} holds for $h=\mu^2-b_1(\lambda) \mu-b_0(\lambda)$ with $h(L_4,M)=0$. Let us prove~2.
On one hand $(M-p_0(L_4))^2$ equals $p_1(L_4)^2X_2^2=p_1(L_4)^2 R_{2g+1}(L_4)$ and on the other it equals
\begin{gather*}
 b_1 M+b_0 +p_0^2-2p_0 M= (b_1-2p_0 )p_1 X_2+b_0 +b_1 p_0 -p_0^2.
\end{gather*}
Thus $ p_1^2 R_{2g+1} =(b_1-2p_0 )p_1 X_2+b_0 +b_1 p_0 -p_0^2$.
But, since $\{1,X_2\}$ is a basis of the free ${\bf C}[L_4]$-module ${\bf C}[L_4,X_2 ]$, it holds that $p_0(L_4)=b_1(L_4)/2$.

3.~In order to have 3, it is enough to compute $\dres(L_4-\lambda,M_1 -\mu)$ taking into account 1 and 2.
\end{proof}

\begin{Remark}One can decide if a nontrivial $M$ of a given order exists in the centralizer of $L_4$, we computed it through a Gr\" unbaum approach \cite{Grun} (solving $[L_4,M]=0$ directly), see examples in Section \ref{sec-contruction}. For certain families of operators in the Weyl algebra $\cC(L_4)\neq {\bf C}[L_4]$ it is guaranteed in \cite{DM, Mi1, Mi2, Mo1,Mo2}, see also~\cite{PZ}.
\end{Remark}

\subsection{The algorithm} \label{sec-contruction}

Let $L_4$ be an irreducible operator of order $4$ in $A_1 ({\bf C})$ as in~\eqref{op-weyl}. Let us assume that $\cC(L_4 )\not= \mathbf{C}[L_4 ]$.
By Proposition~\ref{thm-B}, there exists an operator $B_g$ of minimal order $2(2g+1)$, $g\neq 0$, such that
\begin{gather*}
 \mathcal{C}(L_4)={\bf C}[L_4, B_g]\qquad \text{and} \qquad B_g^2=R_{2g+1}(L_4).
\end{gather*}

Now, let us suppose we are given an operator $M$ in the centralizer $\mathcal{C}(L_4)\setminus \mathbf{C}[L_4 ]$. Then $\rk(L_4,M)=2$. The goal of this section is to decide {\it effectively} if $L_4$, $M$ is a BC pair and if not to compute a suitable $B_g$ from $L_4$ and $M$ to have $L_4$, $B_g $ a BC pair; then
\begin{gather*}
 {\bf C}[L_4 ]\subseteq {\bf C}[L_4, M ]\subseteq {\bf C}[L_4, B_g] =\mathcal{C}(L_4).
\end{gather*}
Consequently, by means of the differential resultant (see~\cite{Prev}), we can compute the spectral curve $\Gamma= \operatorname{Spec} (\mathcal{C}(L_4) )$. Moreover, by Corollary \ref{thm-centralizerL4}, the centralizer $\mathcal{C}(L_4)$ is a free ${\bf C} [L_4 ]$-module; hence,
 $M=p_0 (L_4) +p_1 (L_4) B_g$ for some polynomials $p_0, p_1\in {\bf C}[\lambda ]$.

Recall that as $L_4$ and $M$ commute, by Proposition~\ref{thm-B}, they are related by an algebraic equation of the type $\mu^2 -b_1(\lambda) \mu -b_0(\lambda)=0$.
Even if we assume that $M^2=R_{2q+1}(L_4)$, that is $M=p_1 (L_4) B_g$, in general it will not be clear how to identify $p_1(\lambda)$ or $g$ from the factorization of $R_{2q+1}(\lambda)$.

\begin{Remark}\label{remark-Mi}One method to identify $p_1$ would be to compute the roots $\lambda_j$ of $R_{2q+1}(\lambda)$ with multiplicities and then check if $L_4-\lambda_j$ is a factor of $M$. We should observe that factoring $R_{2q+1}(\lambda)$ can generate important problems since the roots can have multiplicity greater than one (since the curve can be singular). In addition, it may not be possible to compute exactly the complex roots of $R_{2q+1}(\lambda)$, this is the case of $R_5(\lambda)$ in \eqref{eq-hB10} of Example~\ref{ex1} or $R_9(\lambda)$ in \eqref{eq-R9} of Example \ref{ex4}. Having approximate roots of the polynomial $R_{2q+1}(\lambda)=h(\lambda,\mu)-\mu^2$ from \eqref{eq-h2} does not guarantee the correct factorization of the operator~$M$, since the factorization occurs at each point of the spectral curve and this point cannot be in a nearby curve (which would be the case if we consider approximate roots of $R_{2q+1}(\lambda)$). Even if the roots and multiplicities are assumed to be known exactly, the combinatorics of the problem gives multiple choices since the genus $g$ is also a variable in this problem.
\end{Remark}

The next construction is an alternative method to the proposal given in Remark \ref{remark-Mi}. Our goal is to develop a symbolic algorithm whose input is an operator $M$ that commutes with the fixed $L_4$, and whose output is a generator $B\not= L_4$ of the centralizer $\mathcal{C}(L_4)$ and the genus $g$ of the spectral curve $\Gamma$.
One of the achievements of this construction is the determination of the genus of the spectral curve associated with the operator $L_4$, in both the self-adjoint and non self-adjoint cases, starting with any operator $M$ that commutes with $L_4$.

\smallskip

{\bf The construction.} From now on we assume that $M=p_1 (L_4) B_g$ of order $m=2(2q+1)$, $q> 0$, and also that $p_1 (0)=1$, see Proposition~\ref{thm-B}.
We will fix a value of $g$ from $1$ to $q-1$ and check if an operator $B_g$ of order $2(2g+1)$ exists in $\cC(L_4)$. Moreover, if such $B_g$ does not exist for $g=1,\ldots ,q-1$, then we conclude that~$L_4$, $M$ is a BC pair, that is $\mathcal{C}(L_4)={\bf C}[L_4,M]$ and $B_g=M$, with $g=q$.

The procedure to obtain $B_g$ is based on an iterated division process. Observe that the ring of differential operators $K[\partial]$ is a (left) Euclidean domain that contains $A_1({\bf C})$, with $K={\bf C}(x)$. Moreover, we will use the construction of a system of equations for a family of free parameters $ \vec{a} = (a_1, \dots, a_d ) $ for a certain length $ d $ determined by a recursive process. Theorem~\ref{thm-al-2} guarantees that the given construction effectively allows for an explicit operator $ B_g $ verifying the required conditions.

 Recall that $\ord (M)=2 (2q+1)$ with $q>0$. Let us fix $g\in \{1,\ldots ,q-1\}$.
We use the left division algorithm in $K[\partial]$ to construct a sequence of quotients and remainders to rewrite $M$ as follows.
First, by left division by~$L_4$, we compute the remainder sequence
\begin{gather}\label{eq-divisionForM}
\Delta (M) =\{\R_1,\ldots , \R_{g+1}\} , \end{gather}
where
\[ M=L_4 \Q_{1} +\R_{1} ,\qquad \Q_{j}=L_4 \Q_{j+1} +\R_{j+1 } ,\qquad 1 \leq j \leq g,\]
with bounds for the orders of the remainders $ \ord ( \R_j ) \leq 3 $, and $\ord (\Q_{g})=4(q-g)-2$. Thus we decompose $M$ as
\begin{gather*}
 M= \sum_{j=0}^{g} L_4^j \R_{j+1} + L_4^{g+1} \Q_{g+1} .
\end{gather*}
Observe that $\R_1,\ldots , \R_{g+1}$ are thus known differential operators in~$K[\partial]$ for the given~$M$.

Recall that we are looking for $B_g$, that could be decomposed using left division by $L_4$ as $B_g= \sum\limits_{j=0}^{g-1} L_4^j \R_{j+1,B} + L_4^g \Q_{g B}$, with $\ord(\R_{j,B})\leq 3$ and $\ord(\Q_{g,B})=2$. Thus, we are looking for $\R_{j+1,B}$, $j=0,\ldots ,g-1$ and $\Q_{g,B}$ in $K[\partial]$.

With this purpose, for the fixed $g\in \{1,\ldots ,q-1\}$ let us consider a vector $\vec{a}= (a_1 , a_2 , \dots , a_{d(g)})$ of free parameters over ${\bf C}$ that will be used to define an extended remainder sequence \begin{gather*}\Delta_{\vec{a}}^g = \left\{ R_{1,B}, \dots , R_{g,B} , Q_{g,B}\right\}\end{gather*}
of operators in $K[\vec{a}][\partial]$ assumed to be of order less than $4$.
Let us define the polynomial
\begin{gather*}
 p_{\vec{a}}(\lambda)=
 1+a_1\lambda+\cdots +a_{d(g)} \lambda^{d(g)}+\lambda q(\lambda), \qquad \text{where} \qquad d(g):=\min\{q-g,g\}
\end{gather*}
for a polynomial $q(\lambda)\in {\bf C}[\lambda]$ which is taken to be equal to zero if $q-g< g$,
and the operator
\begin{gather} \label{eq-Ba}
 B_{\vec{a}}^g: =L_4^g Q_{g, B} +\sum_{j=0}^{g-1} L_4^j R_{j+1,B} \qquad \text{in} \qquad K[\vec{a}] [\partial ].
\end{gather}
Forcing now $M=p_{\vec{a}} (L_4) B_{\vec{a}}^g$, since $\ord (R_{j,B})$ and $\ord (\R_j)$ are smaller than the order of~$L_4$, comparing the terms in $L_4^j$, $j=0,\ldots ,g-1$ we obtain
\begin{gather}\label{eq-recursionR}
\R_1=R_{1,B},\qquad
 \R_{j+1} =
 \begin{cases}
 R_{j+1, B} +a_1 R_{j,B}+\cdots + a_{j}R_{1,B} & \text{if } 0<j<d(g),\\
 R_{j+1, B} +a_1 R_{j,B}+\cdots + a_{d(g)}R_{j+1-d(g),B} & \text{if } j\geq d(g).
 \end{cases}
\end{gather}
From the term in $L_4^g$
\begin{gather}\label{eq-recursionQ}
 \R_{g+1} = Q_{g,B}+a_1R_{g,B} +a_2 R_{g-1,B}+\cdots + a_{d(g)}R_{g-d(g)+1,B}.
\end{gather}

Thus from \eqref{eq-recursionR} and \eqref{eq-recursionQ} we obtain the extended remainder sequence $\Delta_{\vec{a}}^g $ whose operators we now define as
\begin{gather}
R_{1,B}:= \R_1, \nonumber\\
R_{j,B} := \R_j - \begin{cases}
 \left(a_1 R_{j-1,B}+\cdots + a_{j-1}R_{1,B}\right) & \text{if } j\leq d(g),\\
 \left(a_1 R_{j-1,B}+\cdots + a_{d(g)}R_{j-d(g),B}\right) & \text{if } j > d(g),\end{cases}
\qquad \text{for} \quad j=2, \dots ,g, \nonumber\\
Q_{g,B} := \R_{g+1} - \left( a_1R_{g,B} +a_2 R_{g-1,B}+\cdots + a_{d(g)}R_{g-d(g)+1,B} \right) . \label{eq-RB}
\end{gather}
Observe that the order of each $R_{j,B}$ is at most $3$, each $R_{j,B}$ belongs to $K[ a_1 , \dots ,a_{j-1}] [\partial ]$ and $Q_{g,B} \in K [ \vec{a}] [\partial ]$.

Finally, to determine if $B_g$ exists, we look for $\vec{\alpha}=(\alpha_1 , \dots ,\alpha_{d(g)})\in {\bf C}^{d(g)}$
such that $M$ equals $p_{\vec{\alpha}} (L_4) B_{\vec{\alpha}}^g$ and $[L_4, B_{\vec{\alpha}}^g]=0$, where $p_{\vec{\alpha}}$ and $B_{\vec{\alpha}}^g$ are obtained by replacing $\vec{a}$ by $\vec{\alpha}$ in $p_{\vec{a}}$ and $B_{\vec{a}}^g$ respectively.
Thus, forcing
\begin{gather*}\big[L_4, B_{\vec{a}}^g \big]=0\end{gather*}
the parameters $\vec{a}$ can be adjusted.
Observe that the numerator $\cN$ of $[L_4 , B_{\vec{a}}^g ]$ is a differential ope\-rator in ${\bf C}[ \vec{a}][x] [\partial]$.
Let us consider the system of equations obtained from the coefficients~$q_{i,j}(\vec{a})$ of $x^i\partial^j$ in $\cN$
\begin{gather}\label{eq-Sa}
 \textsc{s}(\vec{a})_g=\{q_{i,j}(\vec{a})=0\}, \qquad \text{with} \quad q_{i,j}(\vec{a})\in {\bf C}[ \vec{a}].
\end{gather}

This construction proves the next result.

\begin{Theorem}\label{thm-al-2}Let $L_4$ be an irreducible operator of order $4$ in $A_1 ({\bf C})$ as in~\eqref{op-weyl}.
Given an operator $M\in \cC(L_4)$ of order $m=2(2q+1)$, $q> 0$, such that $M^2=R_{2q+1}(L_4)$ and $g\in \{1,\ldots ,q-1\}$, the following statements are equivalent:
\begin{enumerate}\itemsep=0pt
 \item[$1.$] There exists an operator $B_g$ in $\cC(L_4)$ of order $2(2g+1)$ such that $M=p_1(L_4)B_g$, for some $p_1\in {\bf C}[\lambda]$.

 \item[$2.$] There exists $\vec{\alpha}$ in ${\bf C}^{d(g)}$, where $d(g):=\min\{q-g,g\}$, such that $\big[L_4,B_{\vec{\alpha}}^g\big]=0$, or equivalently~$\vec{\alpha}$ is a solution of~$\textsc{s}(\vec{a})_g$.
\end{enumerate}
\end{Theorem}
\begin{proof}
The previous construction guaranties that any $B_g$ in $\cC(L_4)$ such that $M=p_1(L_4)B_g$ has to be of the form~\eqref{eq-Ba}. Therefore, if $[L_4,B_{\vec{a}}^g]=0$ has no solution $\vec{\alpha}$ in ${\bf C}^{d(g)}$ then such $B_g$ does not exist. Conversely, if there exists $\vec{\alpha}$ in ${\bf C}^{d(g)}$ such that $\big[L_4,B_{\vec{\alpha}}^g\big]=0$ then $B_{\vec{\alpha}}^g$ is an operator of order $2(2g+1)$ in $\cC(L_4)$ such that $M=p_{\vec{\alpha}}(L_4)B_{\vec{\alpha}}^g$.
\end{proof}

Let $g^*$ be the minimum of the set of non negative integers
\[
\cG= \big\{ g\in \{1,\ldots ,q-1\} \colon \exists \, \vec{\alpha} =( \alpha_1 ,\dots ,\alpha_{d(g)} ) \in {\bf C}^{d(g)} \textrm{ solution of } \textsc{s}(\vec{a})_g
\big\} .
\]
By Corollary~\ref{thm-centralizerL4}, $\cG$ is a non empty set, and $g^*$ always exists. From the previous theorem we can conclude:

\begin{Corollary} Given an operator $M\in \cC(L_4)$ of order $m=2(2q+1)$, $q> 0$, such that $M^2=R_{2q+1}(L_4)$, the centralizer $\cC(L_4)$ equals ${\bf C}\big[L_4, B_{\vec{\alpha}^*}^{g^*}\big]$, where $\vec{\alpha}^*=( \alpha_1^* ,\dots ,\alpha_{d(g^*)}^* )$ is a~solu\-tion of $\textsc{s}(\vec{a})_{g^*}$.
\end{Corollary}

\begin{Remark}\label{rem-Sa}The number of variables $a_i$ appearing in the system $\textsc{s}(\vec{a})_g$ is equal to $d(g)$, which depends on the fixed values of $q$ and $g$.
If a new variable $a_j$ appears in iteration $g$ of the algorithm, the polynomials of the system are linear in $a_j$.
Furthermore, all the polynomials $q_{i,j}$ in $\textsc{s}(\vec{a})_g$ will have the same structure, which depends on $Q_{g,B}$ (see step~\ref{step-9} of the algorithm), they will have the form $r_0+r_1a_1+\cdots +r_{g+1} p(a_1,\ldots ,a_{d(g)})$, so we solve linearly a subsystem of $g+1$ nonzero polynomials $q_{i,j}$ in $\textsc{s}(\vec{a})_g$ to obtain $\vec{\alpha}_0$ and then check if $\vec{\alpha}_0$ is a solution of $\textsc{s}(\vec{a})_g$. We illustrate this method in Example \ref{ex4}.
\end{Remark}

We automate the previous construction in the following algorithm.

\begin{Algorithm}[\texttt{BC pair}]\quad
\begin{itemize}\itemsep=0pt
\item \underline{\sf Given} $M$ in $\cC(L_4)$.
\item \underline{\sf Compute} $B$ such that $L_4$, $B$ is a BC pair, and its order.
\end{itemize}
\begin{enumerate}\itemsep=0pt
 \item $f:=\dres (L_4-\lambda ,M-\mu )$.
 \item Compute the square free part $h(\lambda ,\mu )=\mu^2-b_1(\lambda) \mu-b_0(\lambda)$ of $f$.
 \item $M:=M-\frac{1}{2}b_1(L_4)$.
 \item If $M=0$ then \underline{\sf return} `{\sf $M$ is a polynomial in ${\bf C}[L_4]$}'.
 \item g:=1.
 \item Compute the remainder sequence $\Delta(M)=\{ \R_1, \R_2\}$ as in \eqref{eq-divisionForM}.
 \item Use $\Delta(M)$ and \eqref{eq-RB} to construct $R_{1,B}$ and $Q_{1,B}$.
 \item $B_{\vec{a}}^g:=L_4 Q_{1,B}+R_{1,B}$ as in \eqref{eq-Ba}.
 \item From $\big[L_4,B_{\vec{a}}^g\big]=0$ compute the system $\textsc{s}(\vec{a})_g$ as in \eqref{eq-Sa}. \label{step-9}
 \item If a solution $\vec{\alpha}$ of $\textsc{s}(\vec{a})_g$ exists then \underline{\sf return} $B_{\vec{\alpha}}^g$ and $2(2g+1)$.
 \item $g:=g+1$. \label{step-genus+1}
 \item If $g=q$ \underline{\sf return} $M$.
 \item Compute the remainder $\R_{g+1}$ as in~\eqref{eq-recursionR} and $\Delta(M):=\Delta(M)\cup \{\R_{g+1}\}$. \label{step-13}
 \item Use $\Delta(M)$ and \eqref{eq-RB} to construct $Q_{g,B}$.
 \item Define $B_{\vec{a}}^g:=L_4^g Q_{g,B}+B_{\vec{a}}^{g-1}$ and go to step~\ref{step-9}.
\end{enumerate}
\end{Algorithm}

We implemented the algorithm in Maple~18 and we used it to compute the next examples.

\begin{Example}\label{ex3}
Let us continue with Example \ref{ex1} and $L_4$ as in \eqref{eq-nonselfadj1}. From a family of operators of order $10$ commuting with $L_4$ we fix $M$
\begin{gather*}
M= \partial^{10}+\big(5x^4+7/2\big)\partial^8+20\big(4x^3+{\rm i}\big)\partial^7+\big(10x^8+14x^4+640x^2+4\big)\partial^6\\
\hphantom{M=}{} +N_5\partial^5+N_4 \partial^4 +N_3 \partial^3 + N_2 \partial^2 + N_1 \partial + N_0,
\end{gather*}
$N_i\in {\bf C}[x]$ (not included due to their length).
We know that $L_4,M$ is a true rank~2 pair. We run the algorithm \texttt{BC pair} to decide if $L_4$, $M$ is a~BC pair. We fix $g=1$:
\begin{itemize}\itemsep=0pt
 \item $\dres(L_4-\lambda,M-\mu)=h(\lambda,\mu)^2$ with
$h(\lambda,\mu)=\mu^2-b_0(\lambda)-b_1(\lambda)\mu$ where
\begin{gather*} b_0(\lambda)=-\lambda^5+(9/4)\lambda^4+(125/2)\lambda^3+(7825/4)\lambda^2+1548 \lambda+1296,\\
 b_1(\lambda)=3\lambda^3+79\lambda+72.\end{gather*}

\item $M:=M-\frac{1}{2}b_1(L_4 )$ is given by~\eqref{eq-B10ex1} in Example~\ref{ex1}.

\item We compute the reminder sequence $\Delta(M)=\{\R_1,\R_2\}$ to obtain the third order operators
\begin{gather*}
 \R_1= 36+72 {\rm i}x-72 {\rm i}{x}^{5}+108 {x}^{4}+72 {x}^{8}+1728 {x}^{2}
 -72 {\rm i} \big( {-}{x}^{6}+12 {\rm i}{x}^{3}-{x}^{2}-8 \big) \partial\\
\hphantom{\R_1=}{} +\big( 72 {x}^{4}+36+504 {\rm i} x \big) {\partial}^{2}+72 {\rm i} {x}^{2}{\partial}^{3},\\
 \R_2= 8 {x}^{2}+16+8 {x}^{6}-8 {\rm i} {x}^{3}-4 {\rm i} \big( {-}{x}^{4}+12 {\rm i} x-1 \big) \partial+8 {x}^{2}{\partial}^{2}+4 {\rm i} {\partial}^{3}.
\end{gather*}

\item We construct $R_{1,B}=\R_1$ and $Q_{1,B}=\R_2-a_1 \R_1$. Then $B_{\vec{a}}^1=L_4 Q_{1,B}+R_{1,B}$. From $\big[L_4,B_{\vec{a}}^1\big]=0$ we obtain the system $\textsc{s}(a_1 )_1$. All the $120$ polynomials $q_{i,j}(a_1)$ in $\textsc{s}(a_1)_1$ have the form $r_0+r_1a_1$.
From the first two equations
\[
-11296+2889216 a_1 =0 , \qquad 219904 -359424 a_1 =0
\]
we obtain $a_1=353/90288$, and substituting in all the remaining $q_{i,j}(a_1)$ we can conclude that the system $\textsc{s}(a_1)_1$ has no solution.
\end{itemize}
Therefore, in step \ref{step-genus+1} $g:=g+1=2=q$ and the algorithm returns $M=B_{10}$, the operator that was defined in~\eqref{eq-B10ex1} of Example~\ref{ex1}. Therefore the centralizer $\cC(L_4)={\bf C}[L_4,B_{10}]={\bf C}[L_4,M]$.

The \texttt{BC pair} Algorithm can be used to check if a given operator $B$ is a generator of the centralizer for~$L_4$ as in~\eqref{operator-L4general}. For instance, in Case~\ref{ex-1-1} of Example~\ref{ex1}, for a given operator $B$ commuting with $L_4$ the algorithm guarantees if $B$ is a generator of the centralizer~$\cC(L_4) $ or not. If it is, then $\cC(L_4)={\bf C}[L_4,B]$. We run the \texttt{BC pair} algorithm for all cases in Example~\ref{ex1}, even if the operator~$L_4$ was non self-adjoint and we obtained: for
$U(x)=0$ and $W(x)=4x^2+w_0$, then $\cC(L_4)={\bf C}[L_4,B_{6}]$, for an operator $B_6$ of order $6$, and for $W(x)=8x^2+w_0$, then $\cC(L_4)={\bf C}[L_4,B_{10}]$; for $U(x)=\pm 4{\rm i}$ and $W(x)=4x^2+w_0$, then $\cC(L_4)={\bf C}[L_4,B_{6}]$; for $U(x)=\pm 8{\rm i}$ and $W(x)=16x^2+w_0$, then $\cC(L_4)={\bf C}[L_4,B_{10}]$;
for $U(x)=\pm 12{\rm i}$ and $W(x)=12x^2+w_0$, then $\cC(L_4)={\bf C}[L_4,B_{10}]$.
\end{Example}

\begin{Example}\label{ex4}We use the next example to illustrate the structure of the system $\textsc{s}(\vec{a})_g$ as explained in Remark~\ref{rem-Sa}. Let us consider the self-adjoint operator
\begin{gather*}
 L_4=\big(\partial^2+x^4+1\big)^2+24x^2.
\end{gather*}
By \cite[Theorem~2]{Og2016}, this operator commutes with an operator of order $4q+2$ with $q\geq g=2$. We fixed $4q+2=18$, that is $q=4$, and computed an operator $M$ of order $18$ in the centrali\-zer~$\cC(L_4)$. We used a Gr\" unbaum's style approach, forcing $[L_4,M_{18}]=0$ for an arbitrary operator~$M_{18}$ of order~$18$. From the family of operators obtained we chose
\begin{gather*}
M=\partial^{18}+9\big(x^4+1\big)\partial^{16}+288x^3\partial^{15}+\big(36x^8+72 x^4+4572 x^2+15\big)\partial^{14}\\
\hphantom{M=}{}+H_14\partial^{14}+\cdots +H_0
\end{gather*}
with $H_i\in {\bf C}[x]$ (not included due to their length) such that $M^2=R_9(\lambda)$, with
\begin{gather}\label{eq-R9}
 R_9(\lambda)=\big(\lambda^5-5\lambda^4+346\lambda^3+854\lambda^2+24917\lambda+222719\big)\big(\lambda^2-23\lambda-58939\big)^2.
\end{gather}
We run the algorithm \texttt{BC pair} for $g=1$, computing $\Delta(M)=\{\R_1,\R_2\}$ and $B_{\vec{a}}^1$ as we did in Example~\ref{ex3}. We can check that the system $\textsc{s}(a_1)_1$ has no solution. Thus we set $g:=2$ and go to step \ref{step-13} of the algorithm:
\begin{itemize}\itemsep=0pt
 \item Compute $\R_3$ and define $\Delta(M)=\{\R_1,\R_2,\R_3\}$,
 \begin{gather*}
 \R_1= -8487216 {x}^{8}+707268 {x}^{6}-17033371 {x}^{4}-253909212 {x}^{2}-5009815\\
\hphantom{\R_1=}{} + \big( {-}101846592 {x}^{3}+4243608 x \big) {\partial}
 + \left( -8487216 {x}^{4}+707268 {x}^{2}-8546155 \right) {{\partial}}^{2},\\
 \R_2= -3312 {x}^{8}-706992 {x}^{6}+111231 {x}^{4}-806352 {x}^{2}-3420417\\
\hphantom{\R_2=}{} + \big( {-}39744 {x}^{3}-4241952 x \big) {\partial}
 + \big( {-}3312 {x}^{4}-706992 {x}^{2}+114543 \big) {{\partial}}^{2},\\
 \R_3= 144 {x}^{8}-288 {x}^{6}-58604 {x}^{4}+4032 {x}^{2}-60188\\
 \hphantom{\R_3=}{}+ \big( 1728 {x}^{3}-1728 x \big) {\partial} + \big( 144 {x}^{4}-288 {x}^{2}-58748 \big) {{\partial}}^{2}.
 \end{gather*}

 \item Construct $Q_{2,B}=\R_1\big(a_1^2-a_2\big)-\R_2 a_1+\R_3$. Define $B_{\vec{a}}^2:=L_4^2 Q_{2,B}+B_{\vec{a}}^1$ and go to step~\ref{step-9}.

 \item From $\big[L_4,B_{\vec{a}}^2\big]=0$ compute the system $\textsc{s}(a_1,a_2)_2$. All the $112$ polynomials $q_{i,j}(a_1,a_2)$ in this system have the form $r_0+r_1 a_1+r_2 (a_1^2-a_2)$, $r_i\in {\bf C}$. Let us take two equations of system $\textsc{s}(a_1,a_2)_2$
 \begin{gather*}
 135795456 {{a_1}}^{2}-52992 { a_1}-135795456 { a_2}-2304=0,\\
 -5658144 {{a_1}}^{2}-5655936 { a_1}+5658144 {a_2}+2304=0.
 \end{gather*}
 Observe that this kind of system can be solved linearly, and its unique solution is $(\alpha_1^*,\alpha_2^*)=(23/58939,-1/58939)$. We can check that $q_{i,j}(\alpha_1^*,\alpha_2^*)=0$ for every equation in system $\textsc{s}(a_1,a_2)_2$. Therefore $(\alpha_1^*,\alpha_2^*)$ is the unique solution of system $\textsc{s}(a_1,a_2)_2$.

 \item The algorithm returns $B_{10}=B_{(\alpha_1^*,\alpha_2^*)}^2$
 \begin{gather*} B_{10}=\partial^{10}+5\big(x^4+1\big)\partial^8+80 x^3\partial^7+10 \big( x^8+2 x^4+66 x^2+1\big)\partial^6\\
 \hphantom{B_{10}=}{} +E_5\partial^5+\cdots +E_0,\end{gather*}
 with $E_i\in {\bf C}[\lambda]$ (not included due to their length), where $B_{10}^2=R_5(L_4)$ with
 \[R_5(\lambda)=\lambda^5-5\lambda^4+346\lambda^3+854\lambda^2+24917\lambda+222719.\]
\end{itemize}
 Therefore $L_4$, $M$ is not a BC pair and we constructed the BC pair $L_4$, $B_{10}$ such that
 \[{\bf C}[L_4,M]\subset {\bf C}[L_4,B_{10}]=\cC(L_4).\]
\end{Example}

\subsection*{Acknowledgements} The authors would like to thank the
 organizers of the conference AMDS2018 that took place in Madrid, for
 giving them the opportunity to collaborate on these topics, of common
 interest for a long time, and finally write this paper together.
The authors would like to thank the anonymous referees who have helped to improve the final version of this work.
S.L.~Rueda is partially supported by Research Group ``Modelos ma\-tem\'aticos no lineales''.

M.A.~Zurro is partially supported by Grupo UCM 910444.

\pdfbookmark[1]{References}{ref}
\LastPageEnding

\end{document}